\documentclass[12pt,a4paper]{amsart}
\usepackage{amssymb,latexsym,graphicx}
\usepackage{color}
\usepackage{colortbl}
\usepackage[cp1252]{inputenc} 
\usepackage[T1]{fontenc}
\usepackage{lmodern}
\usepackage[english]{babel}
\usepackage{enumerate}
\usepackage[section]{placeins}
\usepackage[pagebackref=true]{hyperref}
\hypersetup{pdftex,colorlinks=true,linkcolor=blue,citecolor=blue,urlcolor=blue}
\usepackage{currfile}

\newtheorem{theorem}{Theorem}[section]
\newtheorem{lemma}[theorem]{Lemma}
\newtheorem{proposition}[theorem]{Proposition}
\newtheorem{definition}[theorem]{Definition}

\newtheorem{remark}[theorem]{Remark}
\newtheorem{remarks}[theorem]{Remarks}

\newtheorem{state}[theorem]{Statement}

\numberwithin{equation}{section}
\numberwithin{figure}{section}
\numberwithin{table}{section}

\definecolor{purple}{RGB}{127,0,255}

\newcommand{\cfr}{\color{red}}
\newcommand{\cfb}{\color{blue}}

\definecolor{lgray}{gray}{0.90}
\newcommand{\cllg}{\cellcolor{lgray}}

\newcommand{\noib}{\noindent $\bullet$}
\newcommand{\noid}{\noindent $\diamond$}

\newcommand{\N}{\mathbb{N}}
\newcommand{\Nb}{\mathbb{N^{\bullet}}}
\newcommand{\R}{\mathbb{R}}

\newcommand{\cD}{\mathcal{D}}
\newcommand{\cE}{\mathcal{E}}
\newcommand{\cH}{\mathcal{H}}
\newcommand{\cL}{\mathcal{L}}

\newcommand{\cP}{\mathcal{P}}
\newcommand{\cQ}{\mathcal{Q}}
\newcommand{\cR}{\mathcal{R}}
\newcommand{\cS}{\mathcal{S}}
\newcommand{\cT}{\mathcal{T}}
\newcommand{\cZ}{\mathcal{Z}}

\newcommand{\mf}{\mathfrak}

\newcommand{\rp}{\R\mathrm{P}}

\newcommand{\hlambda}{\hat{\lambda}}
\newcommand{\rhe}{{\cR}h_e}
\newcommand{\ecp}{\mathrm{\textsc{ECP}}}

\newcommand{\ps}[1]{\langle #1 \rangle}

\newcommand{\sm}{\!\setminus\!}
\newcommand{\poplus}{\overset{\bot}{\oplus}}
\newcommand{\pbigoplus}{\overset{\bot}{\bigoplus}}
\newcommand{\vsp}{\phantom{\big|}}

\DeclareMathOperator{\spc}{sp}
\DeclareMathOperator{\mtp}{mult}
\DeclareMathOperator{\ima}{img}

\begin{document}

\title[Courant nodal domain property]{On Courant's nodal domain property for\\ linear combinations of eigenfunctions\\ Part~{II}}

\author[P. B\'{e}rard]{Pierre B\'erard}
\author[B. Helffer]{Bernard Helffer}

\address{PB: Universit\'{e} Grenoble Alpes and CNRS\\
Institut Fourier, CS 40700\\ 38058 Grenoble cedex 9, France.}
\email{pierrehberard@gmail.com}

\address{BH: Laboratoire Jean Leray, Universit\'{e} de Nantes and CNRS\\
F44322 Nantes Cedex, France, and LMO, Universit\'e Paris-Sud.}
\email{Bernard.Helffer@univ-nantes.fr}

\date{\today ~(\currfilename)}

\dedicatory{To Erik Balslev, in memoriam}

\thanks{ The authors express their hearty thanks to V.~Bonnaillie-No\"{e}l for providing numerical simulations at an earlier stage of their research.}



\begin{abstract}
Generalizing Courant's nodal domain theorem, the ``Extended Courant property'' is the statement that a linear combination of the first $n$ eigenfunctions has at most $n$ nodal domains. In a previous paper (Documenta Mathematica, 2018, Vol. 23, pp. 1561--1585), we gave simple counterexamples to this property, including convex domains. In the present paper,  using some input from numerical computations, we pursue the investigation of the Extended Courant property with two new examples, the equilateral rhombus and the regular hexagon.
\end{abstract}%

\maketitle

\section[Introduction]{Introduction}\label{S-intro}

\subsection{Notation}\label{SS-intro-1}

Let $\Omega \subset \R^2$ be a piecewise smooth bounded open domain (we will actually only work with convex polygonal domains), with boundary $\partial \Omega = \overline{\Gamma_{1} \sqcup \Gamma_{2}}$, where $\Gamma_1, \Gamma_2$ are two disjoint open subsets of $\partial \Omega$. We consider the eigenvalue problem
\begin{equation}\label{E-int-2}
\left\{
\begin{array}{rll}
- \Delta u &= \mu \, u &\text{in~} \Omega \,, \\[5pt]
u &= 0 &\text{on~} \Gamma_{1}\,, \\[5pt]
\nu \cdot u &= 0 &\text{on~} \Gamma_{2} \,,
\end{array}%
\right.
\end{equation}
where $\nu$ is the outer unit normal along $\partial \Omega$ (defined almost everywhere).\medskip

Let $\lbrace \mu_i(\Omega, \mf{dn}), i \ge 1 \rbrace$ (resp. $\spc(\Omega, \mf{dn})$) denote the eigenvalues (resp. the spectrum) of problem \eqref{E-int-2}. We always list the eigenvalues in non-decreasing order, with multiplicities, starting with the index $1$.
We simply write $\mu_i$, and skip mentioning the domain $\Omega$, or the boundary condition $\mf{dn}$, whenever the context is clear. Examples of eigenvalue problems with mixed boundary conditions appear in Sections~\ref{S-rhome} and \ref{S-hex}. \medskip

Let $\cE\left( \mu \right)$ denote the eigenspace associated with the eigenvalue $\mu$. \medskip

Define the \emph{min-index} $\kappa(\mu)$ of the eigenvalue $\mu$ as
\begin{equation}\label{E-int-8a}
\kappa(\mu) = \min \left\lbrace m ~|~ \mu = \mu_m \right\rbrace \,.
\end{equation}

\subsection{Courant's nodal domain theorem}\label{SS-int-2}

Let $\phi$ be an eigenfunction of \eqref{E-int-2}. The \emph{nodal set} $\cZ(\phi)$ of $\phi$ is defined as the closure of the set of (interior) zeros of $\phi$,
\begin{equation}\label{E-int-4}
\cZ(\phi) := \overline{\left\lbrace x \in \Omega ~|~ \phi(x) = 0 \right\rbrace}\,.
\end{equation}

A \emph{nodal domain} of $\phi$ is a connected component of the set $\Omega \sm \cZ(\phi)$. Call $\beta_0(\phi)$ the number of nodal domains of $\phi$. We recall the following classical theorem, \cite[Chap.~{VI.6}]{CH1953}.

\begin{theorem}[Courant, 1923]\label{T-int-2}
Assume that the eigenvalues of \eqref{E-int-2} are listed in non-decreasing order, with multiplicities,
\begin{equation}\label{E-int-6}
\mu_1 < \mu_2 \le \mu_3 \le \cdots \,.
\end{equation}
Then, for any eigenfunction $\phi \in \cE(\mu)$ of \eqref{E-int-2}, associated with the eigenvalue $\mu$,
\begin{equation}\label{E-int-6a}
\beta_0(\phi) \le \kappa(\mu) \,.
\end{equation}
In particular,  any $\phi \in \cE(\mu_k)$ has a most $k$ nodal domains,
\end{theorem}%

Courant's theorem is a partial generalization, to higher dimensions, of a classical theorem of C.~Sturm (1836). Indeed, in dimension $1$, a $k$-th eigenfunction of the Sturm-Liouville operator $-\frac{d^2}{dx^2} + q(x)$ in $]a,b[$, with Dirichlet, Neumann, or mixed Dirichlet-Neumann boundary condition at $\{a,b\}$, has exactly $k$ nodal domains in $]a,b[$. In dimension $2$ (or higher), Courant's theorem is not sharp. On the one hand, A.~Stern (1925) proved that for the square with Dirichlet boundary condition, or for the $2$-sphere, there exist eigenfunctions of arbitrarily high energy, with exactly two or three nodal domains. On the other hand, {\AA}.~Pleijel (1956) proved that, for any bounded domain in $\R^2$,  there are only finitely many Dirichlet eigenvalues for which Courant's theorem is sharp. We refer to \cite{BH2015-tsg,Ple1956} for more details, and to \cite{Len2016} for Pleijel's estimate under Neumann boundary condition.\medskip

Another remarkable theorem of Sturm states that any non trivial linear combination $u = \sum_{k=m}^{n} a_j u_j$ of eigenfunctions of the operator $-\frac{d^2}{dx^2} + q(x)$ has at most $(n-1)$ zeros (counted with multiplicities), and at least $(m-1)$ sign changes in the interval $]a,b[$, see \cite{BH2018-sturm}.\medskip

A footnote in \cite[p.~454]{CH1953} states that Courant's theorem \emph{may be generalized as follows: Any linear combination of the first $n$ eigenfunctions divides the domain, by means of its nodes, into no more than $n$ subdomains. See the G\"{o}ttingen dissertation of H.~Herrmann, Beitr\"{a}ge zur Theorie der Eigenwerten und Eigenfunktionen, 1932.}\medskip

For later reference, we introduce the following definition.

\begin{definition}\label{D-iecp-2}
We say that the \emph{Extended Courant property} is true for the eigenvalue problem $(\Omega,\mf{b})$, or simply that the $\ecp(\Omega,\mf{b})$ is true, if,
for any $m \ge 1$, and for any linear combination $v = \sum_{\mu_j \le \mu_m} u_{\mu_j}$, with $u_{\mu_j} \in \cE\big(\mu_j(\Omega,\mf{b})\big)$,
\begin{equation}\label{E-iecp-4}
\beta_0(v) \le \kappa(\mu_m) \le m\,.
\end{equation}
\end{definition}%

The footnote statement in the book of Courant and Hilbert, amounts to saying that $\ecp(\Omega)$ is true for any bounded domain. Already in 1956, Pleijel \cite[p.~550]{Ple1956} mentioned that he could not find a proof of this statement in the literature. In 1973, V.~Arnold \cite{Arn1973,Arn2014} related the statement in Courant-Hilbert to Hilbert's 16th problem. Indeed, should $\ecp(\rp^N,g_0)$ be true (where $g_0$ is the usual metric), then the complement of any algebraic hypersurface of degree $n$ in $\rp^N$ would have at most $\binom{N}{N+n-2}+1$ connected components. Arnold pointed out that while $\ecp(\rp^2,g_0)$ is indeed true, $\ecp(\rp^3,g_0)$ is false due to counterexamples produced by O.~Viro \cite{Vir1979}. More recently, Gladwell and Zhu \cite[p.~276]{GZ2003} remarked that Herrmann in his dissertation and subsequent publications \emph{had not even stated, let alone proved} the $\ecp$. They also produced some numerical evidence that the $\ecp$ is false for non-convex domains in $\R^2$ with the Dirichlet boundary condition, and conjectured that it is true for convex domains.\medskip

Our motivations to look into the Extended Courant property came from reading the papers \cite{Arn2011,GZ2003,Kuz2015}. In \cite{BH2018-ecp1}, we gave simple counterexamples to the $\ecp$ for domains with the Dirichlet or the Neumann boundary conditions (equilateral triangle, hypercubes, domains and surfaces with cracks). This was made possible by the fact that the eigenvalues and eigenfunctions of these domains are known explicitly. In \cite{BH2018-teqa}, we proved that $\ecp(\Omega,\mf{n})$ is false for a continuous family of smooth convex domains in $\R^2$, with the symmetries of, and close to the equilateral triangle.\medskip

In the present paper, we continue our investigations of the Extended Courant property by studying two examples, the equilateral rhombus $\rhe$ and the regular hexagon $\cH$, which are related to the equilateral triangle. The eigenvalues and eigenfunctions of these domains are not known explicitly (except for a small subset of them). Using the symmetries of these domains,  and some input from numerical computations, we are able to describe the nodal patterns of the first eigenfunctions, and conclude that the equilateral rhombus and the regular hexagon provide counterexamples to the $\ecp$. \medskip

The paper is organized as follows. In Section~\ref{S-rhome}, we analyze the structure of the first eigenvalues and eigenfunctions of the equilateral rhombus $\rhe$ with either the Neumann or the Dirichlet boundary condition. Subsections~\ref{SS-rhome-p}, \ref{SS-refpr} and \ref{SS-use} provide technical ideas which are used in Section~\ref{S-hex} as well. In Subsection~\ref{SS-rhe-ecp-n}, we prove that $\ecp(\rhe,\mf{n})$ is false. In Subsection~\ref{SS-rhe-ecp-d}, we give numerical evidence that $\ecp(\rhe,\mf{d})$ is false as well.  In Section~\ref{S-hex}, we analyze the structure of the first eigenvalues and eigenfunctions of the regular hexagon $\cH$ with either the Neumann or the Dirichlet boundary condition. In Subsection~\ref{SS-ch-ecp-d}, we give numerical evidence that $\ecp(\cH,\mf{d})$ is false. In Subsection~\ref{SS-ch-ecp-n}, we give numerical evidence that $\ecp(\cH,\mf{n})$ is false.  In Section~\ref{S-final}, we explain our numerical approach, and we make some final remarks and conjectures.

\section{The equilateral rhombus}\label{S-rhome}

\subsection{Symmetries and spectra}\label{SS-rhome-p}

In this subsection, we analyze how symmetries influence the structure of the eigenvalues and eigenfunctions. The analysis is carried out for the equilateral rhombus, but the basics ideas work for the regular hexagon as well, and will be used in Section~\ref{S-hex}.\medskip

In the sequel, we denote by the same letter $L$ a line in $\R^2$, and the mirror symmetry with respect to this line. We denote by $L^*$ the action of the symmetry $L$ on functions, $L^*\phi = \phi\circ L$. \medskip

A function $\phi$ is \emph{even} (or \emph{invariant}) with respect to $L$ if $L^*\phi = \phi$. It is \emph{odd} (or \emph{anti-invariant}) with respect to $L$ if $L^*\phi = -\phi$. In the former case, the line $L$ is an \emph{anti-nodal} line for $\phi$, i.e., the normal derivative $\nu_L \cdot \phi$ is zero along $L$, where $\nu_L$ denotes a unit field normal to $L$ along $L$. In the latter case, the line $L$ is a \emph{nodal} line for $\phi$, i.e., $\phi$ vanishes along $L$.\medskip

Let  $\rhe$ be the  interior of the equilateral rhombus with si\-des of length $1$, and vertices  $(-\frac{\sqrt{3}}{2},0)$, $(0,-\frac{1}{2})$, $(\frac{\sqrt{3}}{2}),0)$ and $(0,\frac{1}{2})$.  Call $D$ and $M$ its diagonals (resp. the longer one and the shorter one). The diagonal $M$ divides the rhombus into two equilateral triangles. The diagonals $D$ and $M$ divide the rhombus into four hemiequilateral triangles. In the sequel, we use the generic notation $\cT_e$ (resp. $\cT_h$) for  any of the equilateral triangles (resp. hemiequilateral triangles) into which the rhombus decomposes, see Figure~\ref{F-rhome-1}.

\vspace{-5mm}
\begin{figure}[!hbt]
\centering
\includegraphics[scale=0.3]{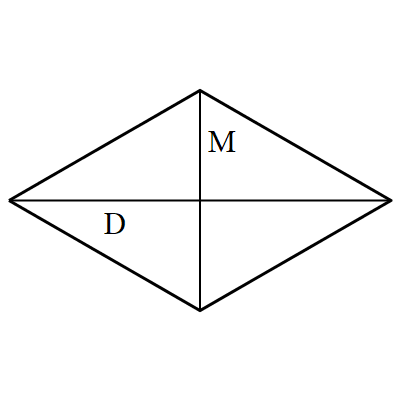}\vspace{-8mm}
\caption{The equilateral rhombus $\rhe$, and its diagonals}\label{F-rhome-1}
\end{figure}


For $L \in \{D,M\}$, define the sets
\begin{equation}\label{E-rhe-2}
\left\{
\begin{array}{ll}
\cS_{L,+} &= \left\lbrace \phi \in L^2(\rhe) ~|~ L^*\phi = + \, \phi \right\rbrace ,\\[5pt]
\cS_{L,-} &= \left\lbrace \phi \in L^2(\rhe) ~|~ L^*\phi = - \, \phi \right\rbrace .
\end{array}%
\right.
\end{equation}
Then, we have the orthogonal decomposition,
\begin{equation}\label{E-rhe-4}
L^2(\rhe) = \cS_{L,+} \poplus \cS_{L,-}\,,
\end{equation}
with respect to the $L^2$-inner product. Indeed, any $\phi \in L^2(\rhe)$ can be decomposed as
\begin{equation}\label{E-rhe-4a}
\phi = \frac 12 (I + L^*)\phi + \frac 12 (I - L^*)\phi \,,
\end{equation}
where $I$ denotes the identity map. \medskip

The symmetries $D$ and $M$ commute
\begin{equation}\label{E-rhe-6}
M \circ D = D \circ M = R_{\pi}\,,
\end{equation}
where $R_{\theta}$ denotes the rotation with center $0$ (the center of the rhombus), and angle $\theta$. It follows that $D^*$ leaves the subspaces $\cS_{M,\pm}$ globally invariant, and that $M^*$ leaves the subspaces $\cS_{D,\pm}$ globally invariant. As a consequence, we have the orthogonal decomposition,
\begin{equation}\label{E-rhe-8}
L^2(\rhe)= \cS_{+ , +} \poplus \cS_{- , -} \poplus \cS_{+ , -} \poplus \cS_{- , +} \,,
\end{equation}
where
\begin{equation}\label{E-rhe-10}
\cS_{\sigma , \tau} := \left\lbrace  \phi \in L^2(\rhe) ~|~ D^* \phi = \sigma \, \phi \text{~and~} M^* \phi = \tau \, \phi
\right\rbrace ,
\end{equation}
for $\sigma, \tau \in \{+ \,, -\}\,.$\medskip

Similar decompositions hold for $H^1(\rhe)$ and $H^1_0(\rhe)$, the Sobolev spaces  which are used in the variational presentation of the Neumann (resp. Dirichlet) eigenvalue problem for the rhombus. \medskip

In the following figures, anti-nodal lines are indicated by dashed lines, and nodal lines by solid lines. Figure~\ref{F-hsymS} displays the nodal and anti-nodal lines common to all functions in $H^1(\rhe) \cap \cS_{\sigma , \tau}$, where $\sigma, \tau \in \{+ , - \}$.

\begin{figure}[htb!]
  \centering
  \includegraphics[scale=0.5]{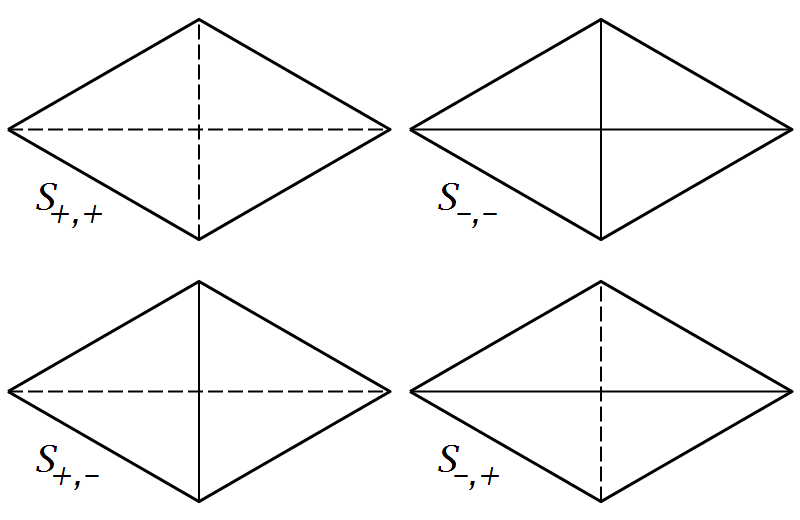}
  \caption{Spaces $\cS_{\sigma,\tau}$ for $\rhe$}\label{F-rheS}
\end{figure}

\FloatBarrier

Because the Laplacian commutes with the isometries $D$ and $M$, the above orthogonal decompositions descend to each eigenspace of $-\Delta$ for $\rhe$, with the boundary condition $\mf{b} \in \{\mf{d,n}\}$ on $\partial\rhe$. The eigenfunctions in each summand correspond to eigenfunctions of $-\Delta$ for the equilateral triangle (decomposition \eqref{E-rhe-4} with $L=M$), or for the hemiequilateral triangle (decomposition \eqref{E-rhe-10}), with the boundary condition $\mf{b}$ on the side supported by $\partial \rhe$, and with mixed boundary conditions, either Dirichlet or Neumann, on the sides supported by the diagonals.\medskip

To  be more explicit, we need naming the eigenvalues as in Subsection~\ref{SS-intro-1}.  For this purpose, we partition the boundaries of $\cT_e$ and $\cT_h$ into their three sides. For $\cT_h$, we number the sides $1$, $2$, $3$, in decreasing order of length, see Figure~\ref{F-Tsides}. For example, $\mu_i(\cT_h,\mathfrak n \mathfrak d \mathfrak n)$ denotes the $i$-th eigenvalue of $-\Delta$ in $\cT_h$ with Neumann boundary condition on the longest (1) and shortest (3) sides, and Dirichlet boundary condition on the other side (2).

\begin{figure}[!hbt]
  \centering
  \includegraphics[scale=0.25]{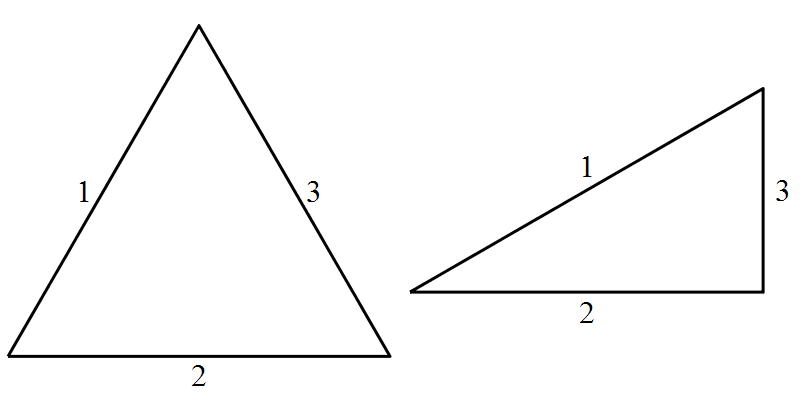}\vspace{-3mm}
  \caption{Labelling the sides of $\cT_e$ and $\cT_h$}\label{F-Tsides}
\end{figure}\medskip

\subsection{Riemann-Schwarz reflection principle}\label{SS-refpr}

In this subsection, we recall the  ``Riemann-Schwarz reflection principle''  which we will use repeatedly in the sequel.\medskip

Consider the decomposition $\rhe = \cT_{e,1} \bigsqcup \cT_{e,2}$, with $M(\cT_{e,1}) = \cT_{e,2}$. Choose a boundary condition $\mf{a} \in \{\mf{d,n}\}$ on $\partial \rhe$. Given an eigenvalue $\lambda$ of $-\Delta$ for $(\rhe,\mf{a})$, and $\sigma \in \{+,-\}$, consider the subspace $\cE(\lambda) \cap \cS_{M,\sigma}$ of eigenfunctions $\phi \in \cE(\lambda)$ such that $M^*\phi = \sigma \phi\,$.\medskip

If $0 \not = \phi \in \cE(\lambda) \cap \cS_{M,\sigma}$, then $\phi|\cT_{e,1}$ is an eigenfunction of $-\Delta$ for $(\cT_{e,1},\mf{aab})$, with $\mf{b}=\mf{n}$ if $\sigma = +$, and $\mf{b}=\mf{d}$ if $\sigma = -$,  associated with the same eigenvalue $\lambda\,$.\smallskip

Conversely, let $\psi$ be an eigenfunction of $(\cT_{e,1},\mf{aab})$, with eigenvalue $\mu_m(\cT_{e,1},\mf{aab})$, for some $m \ge 1$. Define the function $\check{\psi}$ on $\rhe$ such that $\check{\psi}|\cT_{e,1}=\psi$ and $\check{\psi}|\cT_{e,2}=\sigma \, \psi \circ M$. This means that $\check{\psi}$ is obtained by extending $\psi$ across $M$ to $\cT_{e,2}$ by symmetry, in such a way that $M^*\check{\psi} = \sigma \check{\psi}\,$. It is easy to see that the function $\check{\psi}$ is an eigenfunction of $-\Delta$ for $(\rhe,\mf{a})$ (in particular it is smooth in a neighborhood of $M$), with eigenvalue $\mu_m(\cT_{e,1},\mf{aab})$, so that $\check{\psi} \in \cE(\mu_m) \cap \cS_{M,\sigma}\,$.\medskip

The above considerations prove the first two assertions in the following proposition. The proof of the third and fourth assertions is similar, using the symmetries $D$ and $M$, and the decomposition of $\rhe$ into hemiequilateral triangles $\cT_{h,j}, 1\le j \le 4\,$.

\begin{proposition}[Reflection principle]\label{P-refpr-2}
 For any $\mf{a} \in\{\mf{d,n}\}$ and any   $\lambda \in \spc(\rhe,\mf{a})$,
\begin{enumerate}
  \item[(i)] $\cE(\lambda,(\rhe,\mf{a})) \cap \cS_{M,+} \not = \{0\}$ if and only if $\lambda \in \spc(\cT_e,\mf{aan})$, and the map $\phi \mapsto \phi|\cT_{e,1}$ is a bijection from $\cE(\lambda,(\rhe,\mf{a})) \cap \cS_{M,+}$ onto $\cE(\lambda,(\cT_e,\mf{aan}))$;
  \item[(ii)] $\cE(\lambda,(\rhe,\mf{a})) \cap \cS_{M,-} \not = \{0\}$ if and only if $\lambda \in \spc(\cT_e,\mf{aad})$, and the map $\phi \mapsto \phi|\cT_{e,1}$ is a bijection from $\cE(\lambda,(\rhe,\mf{a})) \cap \cS_{M,-}$ onto $\cE(\lambda,(\cT_e,\mf{aad}))$.
\end{enumerate}
More generally, define $\epsilon(\mf{n})=+$ and $\epsilon(\mf{d})=-$. Then, for any $ \lambda \in \spc(\rhe,\mf{a})$, and any $\mf{b,c} \in \{\mf{d,n}\}$,
\begin{enumerate}
  \item[(iii)] $\cE(\lambda,(\rhe,\mf{a})) \cap \cS_{\epsilon(\mf{b}),\epsilon(\mf{c})} \not = \{0\}$ if and only if $\lambda \in \spc(\cT_h,\mf{abc})$,\break and the map $\phi \mapsto \phi|\cT_{h,1}$ is a bijection from $\cE(\lambda,(\rhe,\mf{a})) \cap \cS_{\epsilon(\mf{b}),\epsilon(\mf{c})}$ onto $\cE(\lambda,(\cT_h,\mf{abc}))$.
\end{enumerate}
Furthermore, the multiplicity of the number $\lambda$ as eigenvalue of $(\rhe,\mf{a})$ is the sum, over $\mf{b,c} \in \{\mf{d,n}\}$, of the multiplicities of $\lambda$ as eigenvalue of $(\cT_h,\mf{abc})$ (with the convention that the multiplicity is zero if $\lambda$ is not an eigenvalue).
\end{proposition}%

\subsection{Some useful results}\label{SS-use}

In this subsection, we recall some known results for the reader's convenience.

\subsubsection{Eigenvalue inequalities}

 The following proposition is a particular case of a result of V.~Lotoreichik and J.~Rohleder.

\begin{proposition}[ \cite{LoRo2017}, Proposition~2.3]\label{P-use-2}
Let $\Omega \subset \R^2$ be a polygonal bounded domain whose boundary is decomposed as $\partial \Omega = \overline{\Gamma_{1} \sqcup \Gamma_{2} \sqcup \Gamma_{3}}$, where the $\Gamma_{i}$'s are non-empty open subsets of $\partial \Omega$. Consider the eigenvalue problems for $-\Delta$ in $\Omega$, with the boundary condition $\mf{b}_i \in \{\mf{d,n}\}$ on $\Gamma_{i}$, and list the eigenvalues $\mu_j(\Omega,\mf{b_1b_2b_3})$ in non-decreasing order, with multiplicities, starting from the index $1$.\\[3pt]
Then, for any $j \ge 1$, the following strict inequalities hold.
\begin{equation}\label{E-use-2n}
\left\{
\begin{array}{l}
\mu_j(\Omega,\mf{nnn}) <  \mu_j(\Omega,\mf{ndn}) < \mu_j(\Omega,\mf{ndd}) \,,\\[5pt]
\mu_j(\Omega,\mf{nnn}) <  \mu_j(\Omega,\mf{nnd}) < \mu_j(\Omega,\mf{ndd}) \,,
\end{array}
\right.
\end{equation}
and
\begin{equation}\label{E-use-2d}
\left\{
\begin{array}{l}
\mu_i(\cT_h,\mf{dnn}) < \mu_i(\cT_h,\mf{ddn})  < \mu_i(\cT_h,\mf{ddd}) \,,\\[5pt]
\mu_i(\cT_h,\mf{dnn}) < \mu_i(\cT_h,\mf{dnd}) < \mu_i(\cT_h,\mf{ddd}) \,.
\end{array}
\right.
\end{equation}
\end{proposition}%

 The preceding inequalities can in particular be applied to the triangle $\cT_h$. In this particular case, when $j=1$, we have the following more precise inequalities which are due to B.~Siudeja.

\begin{proposition}[ \cite{Siu2016}, Theorem~1.1]\label{P-use-4}
The eigenvalues of $\cT_h$ with mixed boundary conditions are denoted by $\mu_i(\mf{abc})$, with the sides listed in decreasing order of length. They satisfy the following inequalities.
\begin{equation*}
\begin{array}{ll}
0 = \mu_1(\mf{nnn}) & < \mu_1(\mf{nnd}) < \mu_1(\mf{ndn}) = \mu_2(\mf{nnn}) < \mu_1(\mf{dnn}) \cdots \\[5pt]
&\cdots  < \mu_1(\mf{ndd}) < \mu_1(\mf{dnd}) < \mu_1(\mf{ddn}) < \mu_1(\mf{ddd}) \,.
\end{array}
\end{equation*}
\end{proposition}%

\begin{remark}\label{R-use-P2}
We do not know whether there are any general inequalities between the eigenvalues $\mu_i(\cT_h,\mf{ndn})$ and $\mu_i(\cT_h,\mf{nnd})$, or between the eigenvalues $\mu_i(\cT_h,\mf{ddn})$ and $\mu_i(\cT_h,\mf{dnd})$, for $i \ge 2$.
\end{remark}%
\medskip

\subsubsection{Eigenvalues of some mixed boundary value problems for $\cT_h$} For later reference, we describe the eigenvalues of four mixed eigenvalue problems for $\cT_h$. This description follows easily from \cite{BH2016-lmp} or \cite[Appendix~A]{BH2018-ecp1}.\medskip

The eigenvalues of the equilateral triangle $\cT_e$, with either the Dirichlet or the Neumann boundary condition on $\partial \cT_e$, are the numbers
\begin{equation}\label{E-rhome-u2}
\hlambda(m,n) = \frac{16\pi^2}{9}\,(m^2+mn+n^2)\,,
\end{equation}
 with $(m, n) \in \N\times \N$ for the Neumann boundary condition, and $(m,n) \in \Nb\times\Nb$ for the Dirichlet boundary condition (here $\Nb = \N \sm\{0\}$). The multiplicities are given by,
\begin{equation}\label{E-rhome-u2a}
\mtp(\hlambda_0) = \# \left\lbrace (m,n) \in \cL ~|~ \hlambda(m,n) = \hlambda_0 \right\rbrace \,,
\end{equation}
 with $\cL = \N\times\N$ for the Neumann boundary condition, and $\cL = \Nb\times\Nb$ for the Dirichlet boundary condition.\medskip

 One can associate one or two real eigenfunctions with such a pair $(m,n)$. When $m=n$, there is only one associated eigenfunction, and it is $D$-invariant (here $D$ denotes the bisector of one side of $\cT_e$, see Figure~\ref{F-rhome-1}). When $m \not = n$, there are two associated eigenfunctions, one invariant with respect to $D$, the other one anti-invariant. As a consequence, one can explicitly describe the eigenvalues and eigenfunctions of the four eigenvalue problems $(\cT_h,\mf{nnn})$, $(\cT_h,\mf{ndn})$ (they arise from the Neumann problem for $\cT_e$), and $(\cT_h,\mf{dnd})$, $(\cT_h,\mf{ddd})$ (they arise from the Dirichlet problem for $\cT_e$).\medskip

The resulting eigenvalues are given in Table~\ref{T-rhome-u0}.

\begin{table}[!htb]
\caption{Four mixed eigenvalue problems for $\cT_h$}\label{T-rhome-u0}
\centering
\begin{tabular}[c]{|c|c|}%
\hline
Eigenvalue problem & Eigenvalues\\
\hline
$(\cT_h,\mf{nnn})$ & $\hlambda(m,n)$, for $0 \le m \le n$\\
\hline
$(\cT_h,\mf{ndn})$ & $\hlambda(m,n)$, for $0 \le m < n$\\
\hline
$(\cT_h,\mf{dnd})$ & $\hlambda(m,n)$, for $1 \le m \le n$\\
\hline
$(\cT_h,\mf{ddd})$ & $\hlambda(m,n)$, for $1 \le m < n$\\
\hline
\end{tabular}%
\end{table}

\begin{remark}\label{R-rhome-u4}
As far as we know, there are no such explicit formulas for the eigenvalues of the other mixed boundary value problems for $\cT_h$.
\end{remark}%

Tables~\ref{T-rhome-u0}--\ref{T-rhome-u4} display the first few eigenvalues, the corresponding pairs of integers, and the corresponding indexed eigenvalues for the given mixed boundary value problems for $\cT_h\,$.

\begin{table}[!htb]
\caption{First eigenvalues for $(\cT_h,\mf{nnn})$ and $(\cT_h,\mf{ndn})$}\label{T-rhome-u2}
\centering
\begin{tabular}[c]{|c|c|c|c|}%
\hline
Eigenvalue & Pairs & $(\cT_h,\mf{nnn})$ & $(\cT_h,\mf{ndn})$\\
\hline
$0$ & $(0,0)$ & $\mu_1$ & {} \\
\hline
$\frac{16\pi^2}{9}\vsp$ & $(0,1), (1,0)$ & $\mu_2$ & $\mu_1$ \\
\hline
$3 \times \frac{16\pi^2}{9}\vsp$ & $(1,1)$ & $\mu_3$ & {} \\
\hline
$4 \times \frac{16\pi^2}{9}\vsp$ & $(0,2, (2,0)$ & $\mu_4$ & $\mu_2$ \\
\hline
$7 \times \frac{16\pi^2}{9}\vsp$ & $(1,2), (2,1)$ & $\mu_5$ & $\mu_3$ \\
\hline
$9 \times \frac{16\pi^2}{9}\vsp$ & $(0,3), (3,0)$ & $\mu_6$ & $\mu_4$ \\
\hline
\end{tabular}%
\end{table}

\begin{table}[!htb]
\caption{First eigenvalues for $(\cT_h,\mf{dnd})$ and $(\cT_h,\mf{ddd})$}\label{T-rhome-u4}
\centering
\begin{tabular}[c]{|c|c|c|c|}%
\hline
Eigenvalue & Pairs & $(\cT_h,\mf{dnd})$ & $(\cT_h,\mf{ddd})$\\
\hline
$3 \times \frac{16\pi^2}{9}\vsp$ & $(1,1)$ & $\mu_1$ & {} \\
\hline
$7 \times \frac{16\pi^2}{9}\vsp$ & $(1,2), (2,1)$ & $\mu_2$ & $\mu_1$ \\
\hline
$12 \times \frac{16\pi^2}{9}\vsp$ & $(2,2)$ & $\mu_3$ & {} \\
\hline
$13 \times \frac{16\pi^2}{9}\vsp$ & $(1,3), (3,1$ & $\mu_4$ & $\mu_2$ \\
\hline
$19 \times \frac{16\pi^2}{9}\vsp$ & $(2,3), (3,2)$ & $\mu_5$ & $\mu_3$ \\
\hline
$21 \times \frac{16\pi^2}{9}\vsp$ & $(1,4), (4,1)$ & $\mu_6$ & $\mu_4$ \\
\hline
\end{tabular}%
\end{table}

\begin{remark}\label{R-rhome-u2}
For later reference, we point out that the eigenvalues which appear in Tables~\ref{T-rhome-u2}  and \ref{T-rhome-u4} are simple.
\end{remark}%

\FloatBarrier

\subsection{Rhombus with Neumann boundary condition}\label{SS-rhome-n}

In this subsection, we choose the Neumann boundary condition on  the boundary $\partial \rhe$  of the equilateral rhombus.

\subsubsection{ The first Neumann eigenvalues of $\rhe$}

\begin{proposition}\label{P-rhome-n2}
Let $\nu_i$ denote the eigenvalues of $(\rhe,\mathfrak{n})$. Then,
\begin{equation}\label{E-rhome-n2}
0 = \nu_1 < \nu_2 < \nu_3 = \nu_4 < \nu_5 \le \cdots
\end{equation}
 More precisely,
\begin{enumerate}[(i)]
  \item The second eigenvalue $\nu_2$ is simple and satisfies
  \begin{equation}\label{E-rhome-n4}
  \nu_2 = \mu_1(\cT_h,\mf{nnd})
 = \mu_1(\cT_e,\mf{nnd}) \,.
  \end{equation}
  If $u_2 \in \cE(\nu_2)$, then it is invariant under the symmetry
  $D$, anti-invariant under the symmetry $M$, and $\cZ(u_2)=M\cap \rhe$.\\
  Furthermore, $u_2|\cT_h$ is a first eigenfunction of
  $(\cT_h,\mf{nnd})$, and
  $u_2|\cT_e$ is a first eigenfunction of $(\cT_e,\mf{nnd})$.
  \item  For the eigenspace $\cE(\nu_3)$ we have
\begin{equation}\label{E-rhome-n6}
\left\{
\begin{array}{l}
\dim \left( \cE(\nu_3) \cap  \cS_{+,+}\right) = \dim \left(
\cE(\nu_3) \cap \cS_{-,+}\right) = 1 \,,\\[5pt]
\cE(\nu_3) \cap \cS_{-,-} = \cE(\nu_3) \cap \cS_{+,-} = \{0\}\,.
\end{array}
\right.
\end{equation}
In particular, the eigenspace $\cE(\nu_3)$ is spanned by  two linearly independent functions $u_3$ and $u_4$ which are $M$ invariant, and whose restrictions to $\cT_e$ generate the eigenspace $\cE\left( \nu_2(\cT_e) \right)$.
\end{enumerate}
\end{proposition}%

\proof According to the Reflection principle, Proposition~\ref{P-refpr-2}, the first six eigenvalues of $(\rhe,\mf{n})$ belong to the set
\begin{equation}\label{E-rhome-n7c2}
\left\lbrace  \mu_i(\cT_h,\mf{nab}) \text{~for~} 1 \le i \le 6 \text{~and~} \mf{a,b} \in \{\mf{d,n}\}\right\rbrace \,.
\end{equation}

Among these numbers, the eigenvalues of $(\cT_h,\mf{nnn})$ and $(\cT_h,\mf{ndn})$ are known explicitly, and they are simple, see Table~\ref{T-rhome-u2}.\medskip

Although the eigenvalues and eigenfunctions of $(\cT_h,\mf{nnd})$ and $(\cT_h,\mf{dnn})$ are, as far as we know, not explicitly known, they satisfy some inequalities: the obvious inequalities $\mu_1 < \mu_2 \le \cdots$, and the inequalities provided by Proposition~\ref{P-use-2} (see \cite{LoRo2017}), and Proposition~\ref{P-use-4} (see \cite{Siu2016}).
\medskip

Table~\ref{T-rhome-n7c10} summarizes what we know about the four first eigenvalues of the problems $(\cT_h,\mf{nab})$, for $\mf{a,b} \in \{\mf{d,n}\}$.

In blue the known values, in red the known inequalities (Propositions~\ref{P-use-2} and \ref{P-use-4}).  The gray cells contain the eigenvalues, listed with multiplicities, for which we have no a priori information, except the trivial inequalities (black inequality signs).

\begin{remark}
Note that we only display the first four eigenvalues in each line, because this turns out to be sufficient for our purposes.
\end{remark}

\begin{remark}
The reason why there are white empty cells in the 5th row is explained in Remark~\ref{R-use-P2}.
\end{remark}

\begin{table}[!htb]
\caption{$\rhe$, Neumann boundary condition}\label{T-rhome-n7c10}
\centering
\begin{tabular}[c]{|c|c|c|c|c|c|c|c|c|}%
\hline
$(\sigma,\tau)$ & $(\cT_h,\mf{nab})$ & $\mu_1$ & {} & $\mu_2$ & {} & $\mu_3$ & {} & $\mu_4$ \\
\hline
$(+,+)$ & $(\cT_h,\mf{nnn})$ &\cfb $0$ & $<$ & \cfb $\frac{16\pi^2}{9}$ & $\cfb <$ & $\cfb 3\,\frac{16\pi^2}{9}$ & $\cfb <$ & $\cfb 4\,\frac{16\pi^2}{9}\vsp$\\
\hline
{Prop.~\ref{P-use-2}} & {} & \rotatebox{90}{$\cfr >$} & {} & \rotatebox{90}{$\cfr >$} & {} & \rotatebox{90}{$\cfr >$} & {} & \rotatebox{90}{$\cfr >$}\\
\hline
$(+,-)$ & $(\cT_h,\mf{nnd})$ & {\cllg} & $<$ & {\cllg} & $\le$ & {\cllg} & $\le$ & {\cllg}\\
\hline
{Prop.~\ref{P-use-4}} & {} & \rotatebox{90}{$\cfr >$} & {} & {} & {} & {} & {} & {}\\
\hline
$(-,+)$ & $(\cT_h,\mf{ndn})$ & \cfb $\frac{16\pi^2}{9}$ & $<$ & \cfb $4\,\frac{16\pi^2}{9}$ & $\cfb <$ & \cfb $7\, \frac{16\pi^2}{9}$ & $\cfb <$ & \cfb $9\, \frac{16\pi^2}{9}\vsp$\\
\hline
{Prop.~\ref{P-use-2}} & {} & \rotatebox{90}{$\cfr >$} & {} & \rotatebox{90}{$\cfr >$} & {} & \rotatebox{90}{$\cfr >$} & {} & \rotatebox{90}{$\cfr >$}\\
\hline
$(-,-)$ & $(\cT_h,\mf{ndd})$ & {\cllg} & $<$ & {\cllg} & $\le$ & {\cllg} & $\le$ & {\cllg}\\
\hline
\end{tabular}%
\end{table}%

\noib~ We know that $\nu_1 = 0$, and that this eigenvalue is simple.

\noib~ From Table~\ref{T-rhome-n7c10}, we deduce that
\begin{equation*}
\nu_2 \in \lbrace  \mu_2(\cT_h,\mf{nnn}),\mu_1(\cT_h,\mf{nnd}) \rbrace \,,
\end{equation*}
with no other possibility. On the other hand, $\mu_1(\cT_h,\mf{nnd}) < \mu_1(\cT_h,\mf{ndn})$ $= \mu_2(\cT_h,\mf{nnn})$. It follows that $\nu_2 = \mu_1(\cT_h,\mf{nnd})$, and that this eigenvalue is simple, $\nu_2 < \nu_3$.

\noib~ From Table~\ref{T-rhome-n7c10} and the knowledge of $\nu_1$ and $\nu_2$, we deduce that
\begin{equation*}
\nu_3 \in \lbrace \mu_2(\cT_h,\mf{nnn}), \mu_1(\cT_h,\mf{ndn})  \rbrace \,,
\end{equation*}
with no other possibility. Since $\mu_2(\cT_h,\mf{nnn})= \mu_1(\cT_h,\mf{ndn}) $, we have $\nu_3 = \nu_4 < \nu_5$. The proposition follows. \hfill \qed \medskip

%
%

Note: For the reader's information, Table~\ref{T-rhome-n7c12}, displays
numerical values for the eigenvalues: in the gray cells, the numerical values computed with \textsc{matlab}; in the other cells, the approximate values of the
known eigenvalues.

\begin{table}[!htb]
\caption{$\rhe$, Neumann boundary condition}\label{T-rhome-n7c12}
\centering
\begin{tabular}[c]{|c|c|c|c|c|c|c|c|c|}%
\hline
$(\sigma,\tau)$ & $(\cT_h,\mf{nab})$ & $\mu_1$ & {} & $\mu_2$ & {} & $\mu_3$ & {} & $\mu_4$ \\
\hline
$(+,+)$ & $(\cT_h,\mf{nnn})$ &\cfb $0$ & $<$ & \cfb $17.55$ & $\cfb <$ & $\cfb 52.64$ & $\cfb <$ & $\cfb 70.18$\\
\hline
{} & {} & \rotatebox{90}{$\cfr >$} & {} & \rotatebox{90}{$\cfr >$} & {} & \rotatebox{90}{$\cfr >$} & {} & \rotatebox{90}{$\cfr >$}\\
\hline
$(+,-)$ & $(\cT_h,\mf{nnd})$ & {\cllg $7.16$} & $<$ & {\cllg $37.49$} & $\le$ & {\cllg $90.06$} & $\le$ & {\cllg $120.87$}\\
\hline
{} & {} & \rotatebox{90}{$\cfr >$} & {} & {} & {} & {} & {} & {}\\
\hline
$(-,+)$ & $(\cT_h,\mf{ndn})$ & \cfb $17.55$ & $<$ & \cfb $70.18$ & $\cfb <$ & \cfb $122.82$ & $\cfb <$ & \cfb $157.91$\\
\hline
{} & {} & \rotatebox{90}{$\cfr >$} & {} & \rotatebox{90}{$\cfr >$} & {} & \rotatebox{90}{$\cfr >$} & {} & \rotatebox{90}{$\cfr >$}\\
\hline
$(-,-)$ & $(\cT_h,\mf{ndd})$ & {\cllg $47.63$} & $<$ & {\cllg $110.36$} & $\le$ & {\cllg $189.52$} & $\le$ & {\cllg $224.68$}\\
\hline
\end{tabular}%
\end{table}%
\FloatBarrier

\begin{remark}\label{R-rhome-n2R}
One can also deduce Proposition~\ref{P-rhome-n2} from the proof of Corollary~1.3 in \cite{Siu2016} which establishes that the first four Neumann eigenvalues of a rhombus ${\cR}h(\alpha)$ with smallest angle $2\alpha > \frac{\pi}{3}$ are simple, and describes the nodal patterns of the corresponding eigenvalues. When $2\alpha = \frac{\pi}{3}$ the eigenvalues $\nu_3$ and $\nu_4$ become equal, see also Remarks~4.1 and 4.2 in \cite{Siu2016}.
\end{remark}%

\subsection{$\ecp(\rhe,\mf{n})$ is false}\label{SS-rhe-ecp-n}

As a corollary of Pro\-position~\ref{P-rhome-n2}, we obtain,

\begin{proposition}\label{P-rhome-n4}
The Extended Courant property is false for the equilateral rhombus with Neumann boundary condition. More precisely, there exists a linear combination of eigenfunctions in $\cE(\nu_1) \bigoplus \cE(\nu_3)$ with four nodal domains.
\end{proposition}%

\proof Proposition~\ref{P-rhome-n2}, Assertion~(ii) tells us that $\cE(\nu_3)$ contains an eigenfunction which arises from a second $D$-invariant Neumann eigenfunction of $\cT_{e,1}=\cT_e$.  It suffices to apply the arguments of \cite[Section~3.1]{BH2018-ecp1}, where we prove that $\ecp(\cT_0,\mf{n})$ is false. Here, $\cT_0$ is the equilateral triangle with vertices $(0,0), (1,0)$ and $(\frac{1}{2},\frac{\sqrt{3}}{2})$. A second $D$-invariant Neumann
eigenfunction for $\cT_0$ is given by
\begin{equation}\label{E-rhome-n20}
\phi(x,y)= 2 \cos\left( \frac{2\pi x}{3} \right) \left( \cos\left( \frac{2\pi x}{3} \right) + \cos\left( \frac{2\pi y}{\sqrt{3}} \right) \right)-1.
\end{equation}
The linear combination $\phi + 1$ vanishes on the line segments $\{x = \frac{3}{4}\} \cap \cT_0$ and $\{x + \sqrt{3}\, y = \frac{3}{2}\} \cap \cT_0$.\medskip

Transplant the function $\phi$ to $\cT_{e,1}$ by rotation and, using the symmetry with respect to $M$, extend it to an $M$-invariant eigenfunction $u_3$  for $(\rhe,\mf{n})$. The linear combination $u_3 + 1$ vanishes on two line segments which divide $\rhe$ into four nodal domains, see Figure~\ref{F-rhome-n}. The proposition is proved. \hfill \qed

\vspace{-5mm}
\begin{figure}[!htb]
  \centering
  \includegraphics[scale=0.30]{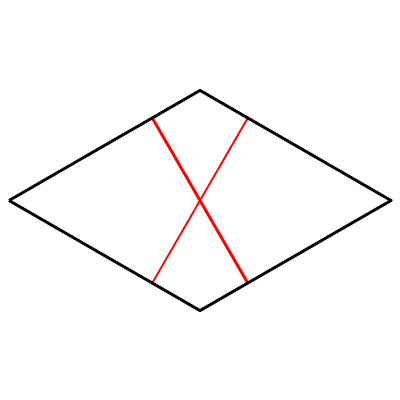}\vspace{-8mm}
  \caption{ Nodal pattern of $u_3 + 1$,  four nodal domains}\label{F-rhome-n}
\end{figure}

Figure~\ref{F-rhome-n2} illustrates the variation of the number of nodal domains (the eigenfunction produced by \textsc{matlab} is proportional to $u_3$, not equal, so that the bifurcation value is not $1$ as in the proof of Proposition~\ref{P-rhome-n4}).

\begin{figure}[!htb]
  \centering
  \includegraphics[scale=0.15]{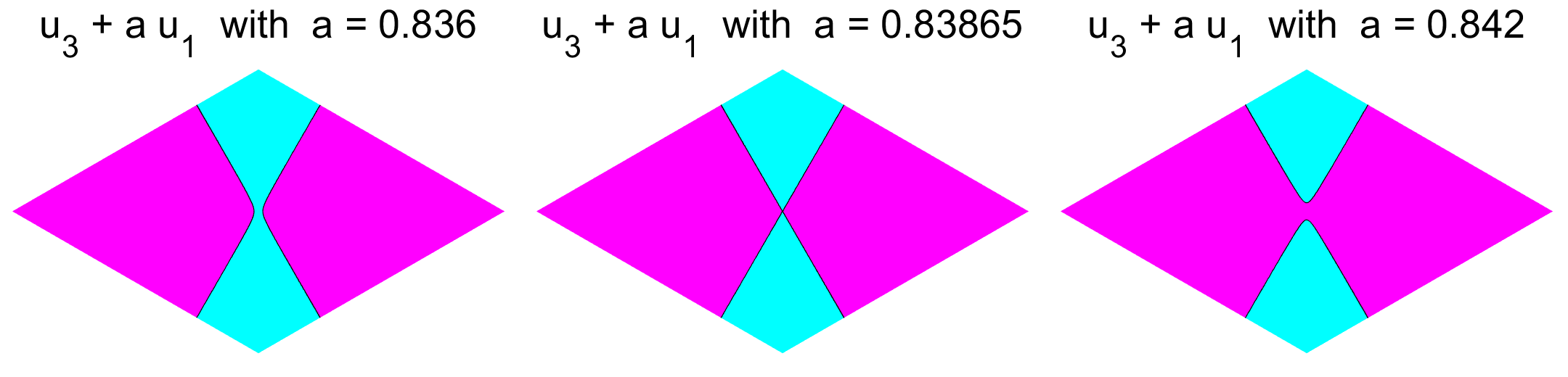}
  \caption{ $(\rhe,\mf{n})$: $\ecp$ false in $\cE(\nu_1) \oplus \cE(\nu_3)$}\label{F-rhome-n2}
\end{figure}

\FloatBarrier

\subsection{Numerical results for the $\ecp(\rhe,\mf{n})$}\label{SS-rhe-n}

 In Subsection~\ref{SS-rhome-n}, we have identified the first four eigenvalues of $(\rhe,\mf{n})$, in particular $\nu_2 = \mu_1(\cT_h,\mf{nnd})$. The numerical computations in Table~\ref{T-rhome-n7c12} indicate that the next eigenvalues are $\nu_5 = \mu_2(\cT_h,\mf{nnd}) = \mu_1(\cT_h,\mf{ndd}) = \nu_6$, so that the Neumann eigenvalues of the equilateral rhombus satisfy,
\begin{equation}\label{E-rhome-n30}
 0 = \nu_1 < \nu_2 < \nu_3 = \nu_4 < \nu_5 < \nu_6 < \, \ldots \,,
\end{equation}
with corresponding nodal patterns shown in Figure~\ref{F-rhome-neu}. Looking at linear combinations of the form $u_5+au_2$, see Figure~\ref{F-rhome-neu2}, we obtain the following numerical result.

\begin{state}\label{S-rhe-ecp2}
Numerical computations of the eigenvalues and of the eigenfunctions indicate that the $\ecp(\rhe,\mf{n})$ is false in $\cE(\nu_2) \oplus \cE(\nu_5)$. More precisely, there exist linear combinations  with six nodal domains.
\end{state}%

\begin{figure}[!htb]
  \centering
  \includegraphics[scale=0.2]{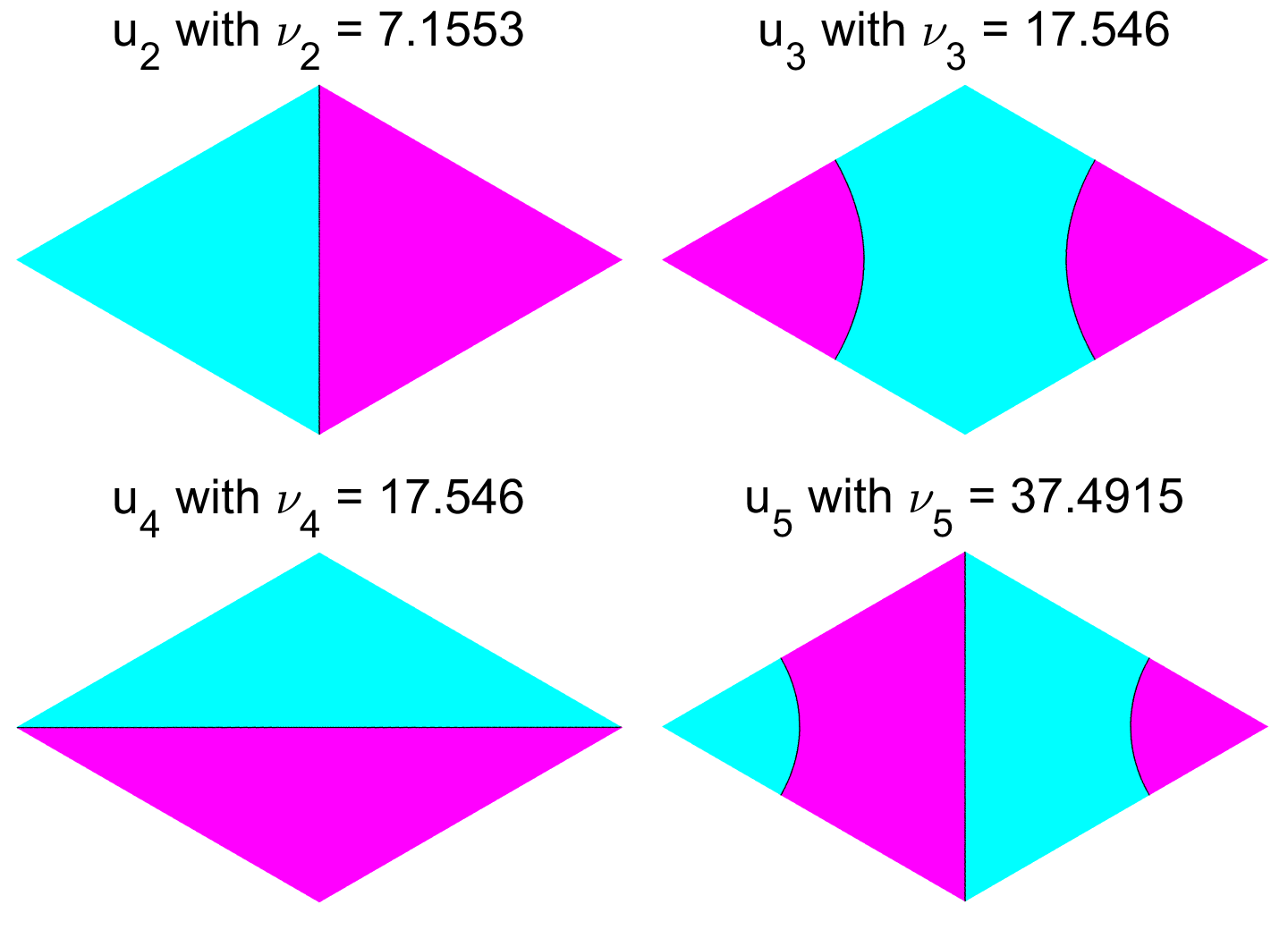}
  \caption{$(\rhe,\mf{n})$: nodal patterns $u_2$ -- $u_5$}\label{F-rhome-neu}
\end{figure}

\begin{figure}[!htb] 
  \centering
  \includegraphics[scale=0.15]{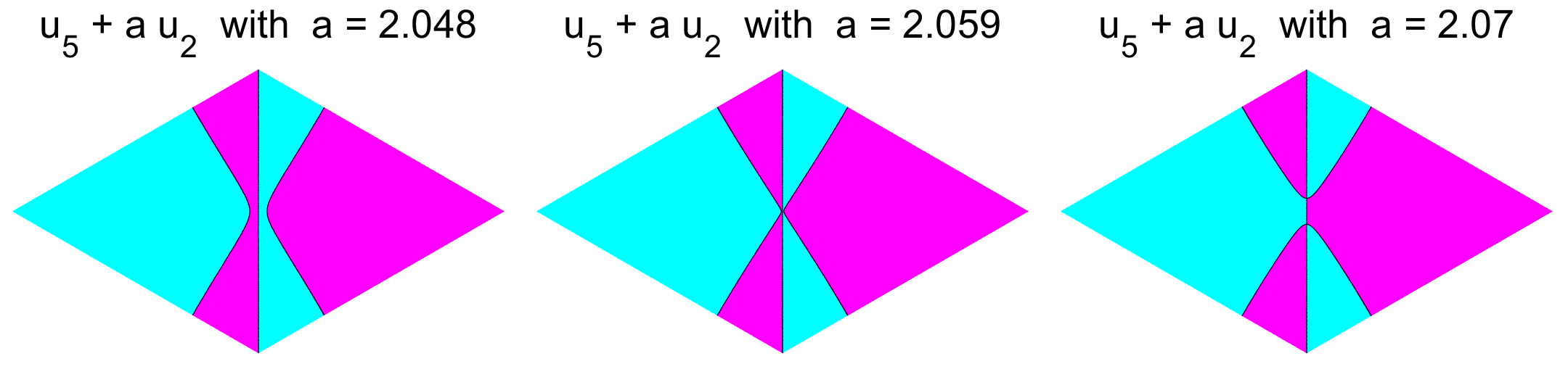}\vspace{-5mm}
  \caption{ $(\rhe,\mf{n})$: $\ecp$ false in $\cE(\nu_2) \oplus \cE(\nu_5)$}\label{F-rhome-neu2}
\end{figure}

\FloatBarrier

\begin{remark}\label{R-rhe-ecp2s}
This counterexample can also be interpreted as a counterexample to the $\ecp$ for the equilateral triangle with mixed boundary conditions, Neumann on two sides, and Dirichlet on the third side. We first look at nodal patterns in $\cE\left( \mu_1(\cT_h,\mf{nnd})\right) \oplus \cE\left( \mu_2(\cT_h,\mf{nnd}) \right)$, see Figure~\ref{F-Th-nnd}.  The corresponding nodal patterns in $\cE\left( \mu_1(\cT_e,\mf{nnd})\right) \oplus \cE\left( \mu_2(\cT_e,\mf{nnd}) \right)$ are obtained using the symmetry with respect to the ho\-ri\-zontal side.
\end{remark}%

\begin{figure}[!htb] 
  \centering
  \includegraphics[scale=0.2]{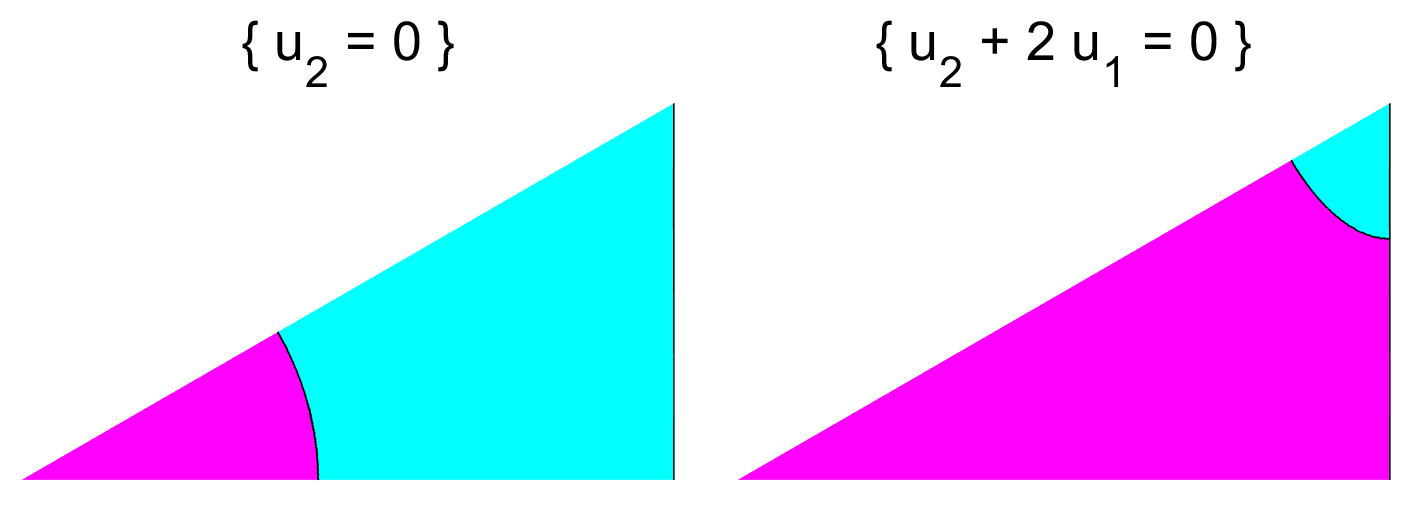}\vspace{-5mm}
  \caption{ $(\cT_h,\mf{nnd})$: nodal patterns in $\cE(\mu_1) \oplus \cE(\mu_2)$}\label{F-Th-nnd}
\end{figure}

\begin{remark}\label{R-rhe-ecp2t}
We refer to Section~\ref{S-final} for comments on our numerical approach.
\end{remark}%
\FloatBarrier


\subsection{Numerical results for the $\ecp(\rhe,\mf{d})$}\label{SS-rhe-ecp-d}%

Table~\ref{T-rhome-d7c10a} is the analogue of Table~\ref{T-rhome-n7c10} for the Dirichlet problem in $\rhe$. Although one can identify the first two Dirichlet eigenvalues of $\rhe$ as $\delta_1(\rhe) = \mu_1(\cT_h,\mf{dnn})$ and $\delta_2(\rhe) = \mu_1(\cT_h,\mf{dnd})$, it is not possible to rigorously identify the following eigenvalues. We have to rely on numerical computations.

\begin{table}[!htb]
\caption{$\rhe$, Dirichlet boundary condition}\label{T-rhome-d7c10a}
\centering
\begin{tabular}[c]{|c|c|c|c|c|c|c|c|c|}%
\hline
$(\sigma,\tau)$ & $(\cT_h,\mf{dab})$ & $\mu_1$ & {} & $\mu_2$ & {} & $\mu_3$ & {} & $\mu_4$ \\
\hline
$(+,+)$ & $(\cT_h,\mf{dnn})$ & {\cllg} & $<$ & {\cllg} & $\le$ & {\cllg} & $\le$ &{\cllg}\\
\hline
{Prop.~\ref{P-use-2}} & {} & \rotatebox{90}{$\cfr >$} & {} & \rotatebox{90}{$\cfr >$} & {} & \rotatebox{90}{$\cfr >$} & {} & \rotatebox{90}{$\cfr >$}\\
\hline
$(+,-)$ & $(\cT_h,\mf{dnd})$ & \cfb $3\,\frac{16\pi^2}{9}$ & $<$ & \cfb $7\,\frac{16\pi^2}{9}$ & $\cfb <$ & \cfb $12\, \frac{16\pi^2}{9}$ & $\cfb <$ & \cfb $13\, \frac{16\pi^2}{9}\vsp$\\
\hline
{Prop.~\ref{P-use-4}} & {} & \rotatebox{90}{$\cfr >$} & {} & {} & {} & {} & {} & {}\\
\hline
$(-,+)$ & $(\cT_h,\mf{ddn})$ & {\cllg} & $<$ & {\cllg} & $\le$ & {\cllg} & $\le$ & {\cllg}\\
\hline
{Prop.~\ref{P-use-2}} & {} & \rotatebox{90}{$\cfr >$} & {} & \rotatebox{90}{$\cfr >$} & {} & \rotatebox{90}{$\cfr >$} & {} & \rotatebox{90}{$\cfr >$}\\
\hline
$(-,-)$ & $(\cT_h,\mf{ddd})$ & \cfb $7\, \frac{16\pi^2}{9}$ & $<$ & \cfb $13\, \frac{16\pi^2}{9}$ & $\cfb <$ & \cfb $19\, \frac{16\pi^2}{9}$ & $\cfb <$ & \cfb $21\, \frac{16\pi^2}{9}\vsp$\\
\hline
\end{tabular}%
\end{table}%

Table~\ref{T-rhome-d7c12} provides the numerical eigenvalues computed with \textsc{matlab}, and numerical approximations of the explicitly known eigenvalues.

\begin{table}[!htb]
\caption{$\rhe$, Dirichlet boundary condition}\label{T-rhome-d7c12}
\centering
\begin{tabular}[c]{|c|c|c|c|c|c|c|c|c|}%
\hline
$(\sigma,\tau)$ & $(\cT_h,\mf{dab})$ & $\mu_1$ & {} & $\mu_2$ & {} & $\mu_3$ & {} & $\mu_4$ \\
\hline
$(+,+)$ & $(\cT_h,\mf{dnn})$ & {\cllg $24.90$} & $<$ & {\cllg $83.83$} & $\le$ & {\cllg $140.50$} & {$\le$} & {\cllg $169.20$}\\
\hline
{} & {} & \rotatebox{90}{$\cfr >$} & {} & \rotatebox{90}{$\cfr >$} & {} & \rotatebox{90}{$\cfr >$} & {} & \rotatebox{90}{$\cfr >$}\\
\hline
$(+,-)$ & $(\cT_h,\mf{dnd})$ & \cfb $52.64$ & $<$ & \cfb $122.82$ & $\cfb <$ & \cfb $210.55$ & $\cfb <$ & \cfb $228.10$\\
\hline
{} & {} & \rotatebox{90}{$\cfr >$} & {} & {} & {} & {} & {} & {}\\
\hline
$(-,+)$ & $(\cT_h,\mf{ddn})$ & {\cllg $71.71$} & $<$ & {\cllg $169.80$} & $\le$ & {\cllg $234.10$} & $\le$ & {\cllg $292.70$}\\
\hline
{} & {} & \rotatebox{90}{$\cfr >$} & {} & \rotatebox{90}{$\cfr >$} & {} & \rotatebox{90}{$\cfr >$} & {} & \rotatebox{90}{$\cfr >$}\\
\hline
$(-,-)$ & $(\cT_h,\mf{ddd})$ & \cfb $122.82$ & $<$ & \cfb $228.10$ & $\cfb <$ & \cfb $333.37$ & $\cfb <$ & \cfb $368.47$\\
\hline
\end{tabular}%
\end{table}%

From Table~\ref{T-rhome-d7c12}, we deduce that the Dirichlet eigenvalues of $\rhe$ satisfy
\begin{equation}\label{E-rhome-d7c18}
0 < \delta_1 < \delta_2 < \delta_3 < \delta_4 < \delta_5 = \delta_6 < \delta_7 \cdots \,.
\end{equation}
More precisely, we find that $\delta_2(\rhe) = \mu_1(\cT_h,\mf{dnd}) =\delta_1(\cT_e)$ (the first Dirichlet eigenvalue of the equilateral triangle $\cT_e$). An eigenfunction $u_2$ associated with $\delta_2(\rhe)$ arises from a first Dirichlet eigenfunction of $\cT_e$. We also find that $\delta_5(\rhe) = \mu_2(\cT_h,\mf{dnd}) = \mu_1(\cT_h,\mf{ddd}) = \delta_2(\cT_e)$. Eigenfunctions associated with $\delta_5(\rhe)$ arise from second Dirichlet eigenfunctions of $\cT_e$, one of them $u_5$ is invariant with respect to $D$, the other is anti-invariant. The nodal patterns of $u_2$ and $u_5$ are given in Figure~\ref{F-rhome-dir} (first and last pictures).\medskip

\begin{figure}[!htb]
  \centering
  \includegraphics[scale=0.2]{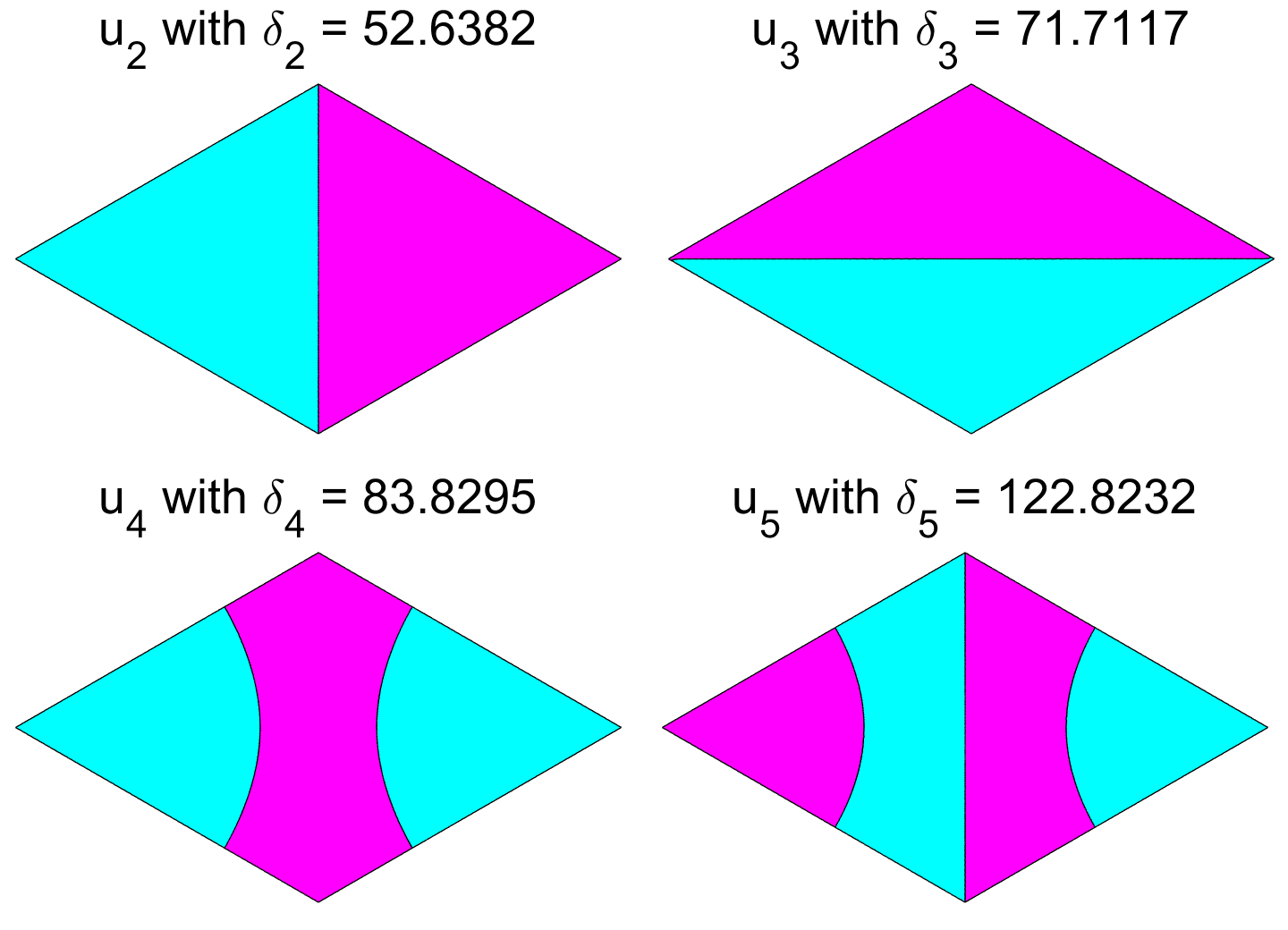}\vspace{-5mm}
  \caption{$(\rhe,\mf{d})$: nodal patterns $u_2$ -- $u_5$}\label{F-rhome-dir}
\end{figure}

In \cite[Section~3]{BH2018-ecp1}, we proved that $\ecp(\cT_e,\mf{d})$ is false: there exists a linear combination of a first eigenfunction and a second $D$-invariant eigenfunction of $(\cT_e,\mf{d})$, with three nodal domains. The same example transcribed to $(\rhe,\mf{d})$ yields a linear combination in $\cE(\delta_2) \oplus \cE(\delta_5)$ with $6$ nodal domains: for the Dirichlet problem in $\rhe$, we have the following (numerical) analogue of Proposition~\ref{P-rhome-n4}, see Figure~\ref{F-rhome-dir2}.

\begin{state}\label{P-rhome-d4n}
The numerical approximations of the eigenvalues $\delta_j(\rhe)$ deduced from Table~\ref{T-rhome-d7c12} indicate that the $\ecp(\rhe,\mf{d})$ is false in $\cE(\delta_2) \oplus \cE(\delta_5)$.
\end{state}

\begin{figure}[!htb]
  \centering
  \includegraphics[scale=0.15]{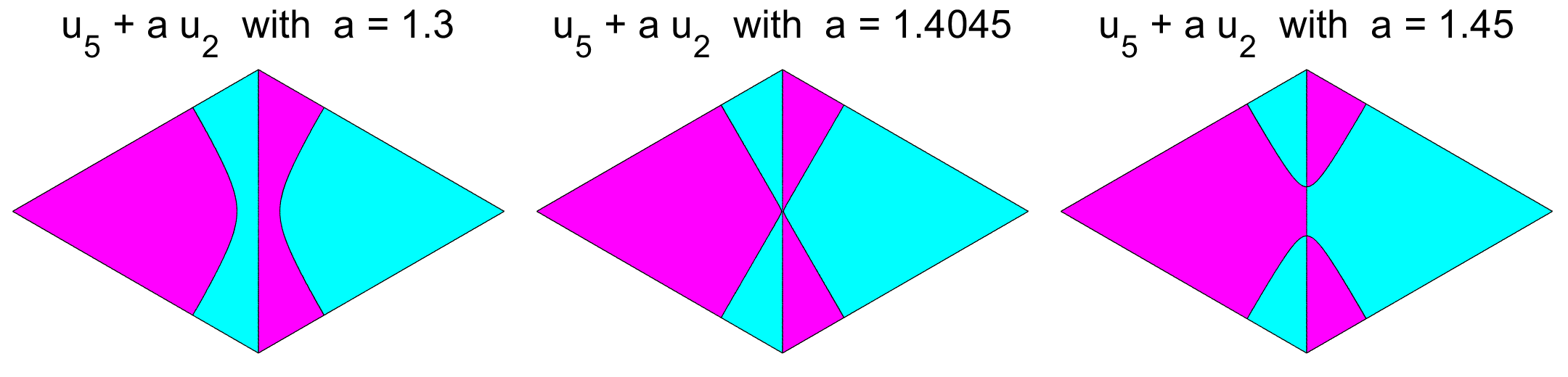}
  \caption{ $(\rhe,\mf{d})$: $\ecp$ false in $\cE(\delta_2) \oplus \cE(\delta_5)$ }\label{F-rhome-dir2}
\end{figure}

\FloatBarrier

\textbf{Remark}. We refer to Section~\ref{S-final} for comments on our numerical approach.

\section{The regular hexagon}\label{S-hex}

\subsection{Symmetries and spectra}\label{SS-ch-p}

Let $\cH$ denote the  interior of the regular hexagon with center at the origin, and sides of unit length. The diagonals $D_i, i=1,2,3$, joining opposite vertices, and the medians $M_j, j= 1, 2, 3$, joining the mid-points of opposite sides, are lines of mirror symmetry of the hexagon $\cH$, see Figure~\ref{F-hsymDM}.\medskip

\begin{figure}[!htb]
\centering
\includegraphics[scale=0.3]{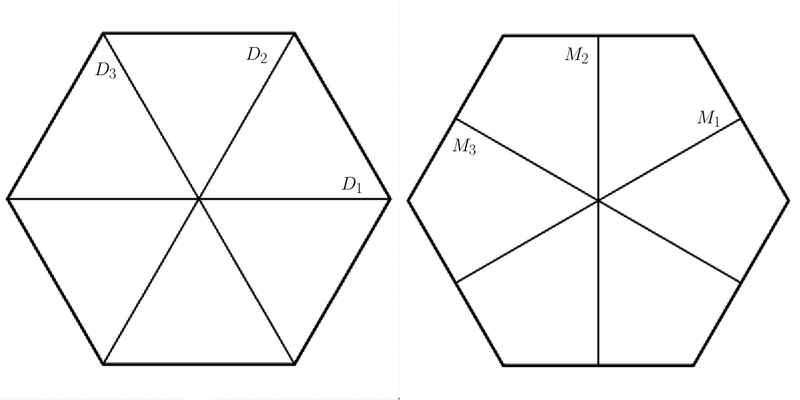}
\caption{The hexagon and its mirror symmetries}\label{F-hsymDM}
\end{figure}

We consider the diagonals $D_1$ and $M_2$, and the associated mirror symmetries of $\cH$. They commute,

\begin{equation}\label{E-hsym-6}
M_2 \circ D_1 = D_1 \circ M_2 = R_{\pi}\,,
\end{equation}
and we can therefore apply the methods of Subsection~\ref{SS-rhome-p}.\medskip

It follows that $D_1^*$ leaves the subspaces $\cS_{M_2,\pm}$ globally invariant, and that $M_2^*$ leaves the subspaces $\cS_{D_1,\pm}$ globally invariant. As a consequence, we have the following orthogonal decomposition of $L^2(\cH)$,
\begin{equation}\label{E-hsym-8}
L^2(\cH)= \cS_{+ , +} \poplus \cS_{- , -} \poplus \cS_{+ , -} \poplus \cS_{- , +} \,,
\end{equation}
where
\begin{equation}\label{E-hsym-10}
\cS_{\sigma , \tau} := \left\lbrace  \phi \in L^2(\cH) ~|~ D_1^* \phi = \sigma \, \phi \text{~and~} M_2^* \phi = \tau \, \phi
\right\rbrace ,
\end{equation}
for $\sigma, \tau \in \{+ \,, -\}\,.$\medskip

Similar decompositions hold for the Sobolev spaces $H^1(\cH)$ and $H^1_0(\cH)$, which are used in the variational presentation of the Neumann (resp. Dirichlet) eigenvalue problem for the hexagon. Since the Laplacian commutes with the isometries $D_1$ and $M_2$, such decompositions also hold for the eigenspaces of $-\Delta$ in $\cH$, with the boundary condition $\mf{b} \in \{\mf{d,n}\}$ on the boundary $\partial \cH$. \medskip

In the following figures, anti-nodal lines are indicated by dashed lines, and nodal lines by solid lines. Figure~\ref{F-hsymS} displays the nodal and anti-nodal lines common to all functions in $H^1(\cH) \cap \cS_{\sigma , \tau}$, where $\sigma, \tau \in \{+ , - \}$.

\begin{figure}[htb!]
  \centering
  \includegraphics[scale=0.5]{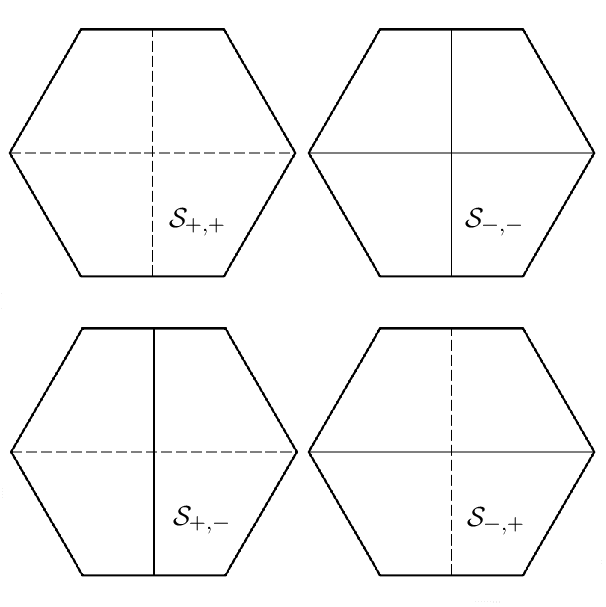}
  \caption{Spaces $\mathcal{S}_{\sigma,\tau}$ for $\sigma,\tau \in \{+,-\}$}\label{F-hsymS}
\end{figure}

\FloatBarrier

Denote by $R$ the rotation $R_{\frac{2\pi}{3}}$,
 \begin{equation}\label{E-hsym-12a}
 \left\{
 \begin{array}{ll}
 R &= D_2 \circ D_1 = M_2 \circ M_1 = \dots \,,\\[5pt]
 R^{-1} &= D_1 \circ D_2 = M_1 \circ M_2 = \dots \,.
 \end{array}
 \right.
 \end{equation}

 This is an isometry of $\cH$, and the action $R^*$ of $R$ on functions is an isometry of $L^2(\cH)$ with respect to the $L^2$-inner-product. \medskip

\begin{lemma}\label{L-hsym-2}
Let
\begin{equation}\label{E-hsym-14}
\left\lbrace
\begin{array}{ll}
\cS^0 & := \ker (R^*-I)\,, \text{and}\\[5pt]
\cS^1 & := \ker (R^{*2} + R^* + I)\,,
\end{array}%
\right.
\end{equation}
as subspaces of $L^2(\cH)$. Then
\begin{equation}\label{E-hsym-16}
\left\{
\begin{array}{ll}
\cS^0 & = \ima \left( R^{*2} + R^* + I \right) =
\ker\left( R^{*2} + R^* + I \right)^{\bot}\,,\\[5pt]
\cS^1 & = \ima (R^* - I) = \ker(R^*-I)^{\bot}\,,\\[5pt]
\end{array}%
\right.
\end{equation}
and we have the orthogonal decomposition
\begin{equation}\label{E-hsym-18}
L^2(\cH) = \cS^0 \poplus \cS^1\,.
\end{equation}
Here, as usual, $\ima(f)$ and $\ker(f)$ denote respectively the image and the kernel of the linear map $f$, and $E^{\bot}$ the subspace orthogonal to $E$.
\end{lemma}%

\proof The following polynomial identities hold.
\begin{equation}\label{E-hsym-20a}
x^3 - 1 = (x-1)(x^2 + x + 1)\,,
\end{equation}
\begin{equation}\label{E-hsym-20b}
3 = (x^2 + x + 1) - (x-1)(x+2)\,.
\end{equation}
Furthermore, the rotation $R$ satisfies
\begin{equation}\label{E-hsym-20c}
R^3 = I\,.
\end{equation}

From \eqref{E-hsym-20a} and \eqref{E-hsym-20c}, we deduce that
\begin{equation}\label{E-hsym-20d}
\ima (R^{*2}+ R^* + I) \subset \ker (R^* - I)\,,
\end{equation}
and
\begin{equation}\label{E-hsym-20e}
\ima (R^* - I) \subset \ker (R^{*2} + R + I)\,.
\end{equation}

From \eqref{E-hsym-20b}, we deduce that
\begin{equation}\label{E-hsym-20f}
L^2(\cH) = \ima (R^* - I) + \ima (R^{*2} + R + I)\,,
\end{equation}
and hence, using \eqref{E-hsym-20d} and \eqref{E-hsym-20e}
\begin{equation}\label{E-hsym-20g}
L^2(\cH) = \ker (R^* - I) + \ker (R^{*2} + R + I)\,.
\end{equation}
Clearly,
\begin{equation}\label{E-hsym-20h}
\ker (R^* - I) \cap \ker (R^{*2} + R + I) = \{0\}\,,
\end{equation}
so that, using \eqref{E-hsym-20d} and \eqref{E-hsym-20e},
\begin{equation}\label{E-hsym-20i}
\ima (R^* - I) \cap \ima (R^{*2} + R + I) = \{0\}\,.
\end{equation}

Let $\phi \in \ima (R^* - I)$ and $\psi \in \ima (R^{*2} + R + I)$. Using the fact that $R^*$ is an isometry and \eqref{E-hsym-20c}, we conclude that $\ps{\phi,\psi} = 0$ (the $L^2$ inner product). Therefore,
\begin{equation}\label{E-hsym-20j}
\ima (R^* - I) = \ima (R^{*2} + R + I)^{\bot}\,.
\end{equation}

From the previous identities, we deduce that
\begin{equation}\label{E-hsym-20k}
L^2(\cH) = \ima (R^* - I) \poplus \ima (R^{*2} + R + I)\,,
\end{equation}
\begin{equation}\label{E-hsym-20m}
L^2(\cH) = \ker (R^* - I) \poplus \ker (R^{*2} + R + I)\,,
\end{equation}
\begin{equation}\label{E-hsym-20n}
\ima (R^* - I) = \ker (R^{*2} + R + I)\,,
\end{equation}
\begin{equation}\label{E-hsym-20p}
\ima (R^{*2} + R + I) = \ker (R^* - I)\,.
\end{equation}

The lemma is proved. \hfill \qed

\begin{lemma}\label{L-hsym-6}
For $\sigma, \tau \in \{+,-\}$, using the notation \eqref{E-hsym-14}, define the subspaces
\begin{equation}\label{E-hsym-L4a}
\left\{
\begin{array}{ll}
\cS^0_{\sigma,\tau} & := \cS_{\sigma,\tau} \cap \cS^0 \,,\\[5pt]
\cS^1_{\sigma,\tau} & := \cS_{\sigma,\tau} \cap \cS^1 \,.
\end{array}
\right.
\end{equation}

Define the map
\begin{equation}\label{E-hsym-L6a}
\left\{
\begin{array}{l}
T : L^2(\cH) \to L^2(\cH)\,,\\[5pt]
T(\phi) = R^*\phi - R^{*2}\phi\,.
\end{array}%
\right.
\end{equation}
Then,
\begin{enumerate}
  \item $\ker(T) = \cS^0$ and $\ker(T)^{\bot} = \cS^1$.
  \item $T^2 = (R^{*2} + R^* + I) - 3 I$\,;\, $T \circ T |\cS^1 = - 3I$\, ;\, $T(\cS^1) = \cS^1$\, ;\, $T$ is a bijection from $\cS^1$ onto $\cS^1$.
  \item $T \circ \Delta = \Delta \circ T$, so that $T$ leaves the eigenspaces of $\Delta$ globally invariant.
  \item For all $\sigma,\tau \in \{+,-\}$, the subspace $\cS^0_{\sigma,\tau}$ satisfies
  \begin{equation}\label{E-hsym-L4b}
\cS^0_{\sigma , \tau} = \left\lbrace \phi \in L^2(\cH) ~|~ D_i^*\phi = \sigma \, \phi \,, M_j^*\phi = \tau \, \phi \,, 1 \le i,j \le 3 \right\rbrace .
\end{equation}
  \item For all $\sigma, \tau \in \{+ , -\}$, $T\left( \cS_{\sigma , \tau}\right) \subset \cS_{-\sigma , -\tau}$.
  \item For all $\sigma, \tau \in \{+ , -\}$, $\ker\left( T|\cS_{\sigma , \tau}\right) = \cS^0_{\sigma , \tau}$, and $\ima(T|\cS_{\sigma,\tau}) \subset \cS^1_{-\sigma,-\tau}$.
   \item For all $\sigma, \tau \in \{+ , -\}$,
   \begin{equation}\label{E-hsym-L4c}
\cS_{\sigma,\tau} =  \cS^0_{\sigma,\tau} \poplus  \cS^1_{\sigma,\tau}\,,
\end{equation}
and $T$ is a bijection from $\cS^1_{\sigma,\tau}$ onto $\cS^1_{-\sigma,-\tau}\,$.
\end{enumerate}
\end{lemma}%

\proof \emph{Assertion~(1)}~ If $\phi \in \ker(T)$, then $R^{*2}\phi = R^*\phi$, so that $\phi = R^{*3}\phi = R^{*2}\phi = R^*\phi$, and $\phi \in \cS^0$. The converse is clear. The second equality follows from Lemma~\ref{L-hsym-2}.\smallskip

\emph{Assertion~(2)}~ The first two equalities are clear. If $\phi \in \cS^1$, then $(R^{*2} + R^* + I) T(\phi) = (R^{*2} + R^* + I)(R^* - I) R^*\phi = 0$, and $T(\phi) \in \cS^1$. If $\phi \in \cS^1$, then $T(T(\phi)) = - 3 \phi$, so that $\phi = T(\psi)$ with $\psi = - \frac 13 T(\phi) \in \cS^1$. This implies that $T(\cS^1) = \cS^1$. On the other hand, if $T(\phi) = 0$ and $\phi \in \cS^1$, then $\phi \in \cS^0 \cap \cS^1 = \{0\}$.\smallskip

\emph{Assertion~(3)}~ This assertion is clear because $R$ is an isometry, so that $R^*$ commutes with $\Delta$. It follows that $T$ commutes with $\Delta$ as well, and hence that $T$ leaves each eigenspace $\cE(\lambda)$ globally invariant. \smallskip

\emph{Assertion~(4)}~  Let $\phi \in \cS^0_{\sigma,\tau}$. Then $R^*\phi = \phi$ and $D_1^*\phi = \sigma \phi$. Since $R = D_1 \circ D_3$, it follows that $\phi = R^* \phi = D_3^* D_1^* \phi = \sigma D_3^* \phi$, so that $D_3^*\phi = \sigma \phi$. The other equalities are established in a similar way. On the other hand, if $D_1^* \phi = D_3^* \phi = \sigma \phi$, then $$R^*\phi = (D_1 \circ D_3)^* \phi = \sigma^2 \phi = \phi\,.$$

\emph{Assertion~(5)}~ Let $\phi \in \cS_{\sigma,\tau}$, i.e., $D_1^*\phi = \sigma \phi$ and $M_2^* \phi = \tau \phi$. Then,
\begin{equation*}
\begin{array}{ll}
D_1^*\left( T(\phi) \right) & = D_1^* R^*\phi - D_1^* R^{*2} \phi\\[3pt]
&= D_1^*(D_2 \circ D_1)^* \phi - D_1^*(D_3 \circ D_1)^* \phi\\[3pt]
&= D_2^*\phi - D_3^*\phi  \\[3pt]
&= (D_1\circ D_1 \circ D_2)^* \phi - (D_1\circ D_1 \circ D_3)^* \phi\\[3pt]
&= (D_1 \circ D_2)^* D_1^* \phi - (D_1 \circ D_3)^* D_1^* \phi\\[3pt]
&= \sigma R^{*2}\phi - \sigma R^* \phi\\[3pt]
&= - \sigma T(\phi) \,.
\end{array}%
\end{equation*}

Similarly, one shows that $M_2^* \left( T(\phi) \right) = - \tau T(\phi)$.\smallskip

\emph{Assertion~(6)}~ The first equality follows from Assertion~(1). The second equality follows from Assertion~(5) and the fact that $\ima(T) \subset \cS^1$ because $R^{*3}=I$.\smallskip

\emph{Assertion~(7)}~ Take $\phi \in \cS_{\sigma,\tau}$. Then $T(\phi) \in \cS^1 \cap \cS_{-\sigma,-\tau}$ and hence $T^2(\phi) \in \cS^1 \cap \cS_{\sigma,\tau}$. We also have $T^2(\phi) = (R^{*2}+ R^* + I)(\phi) - 3 \phi$, which implies that $(R^{*2}+ R^* + I)(\phi) \in \cS^0 \cap \cS_{\sigma,\tau}$. The initial equality can be rewritten $\phi = \frac 13 (R^{*2}+ R^* + I)(\phi) - \frac 13 T^2(\phi)$ which implies that $\cS_{\sigma,\tau} = \cS^0 \cap \cS_{\sigma,\tau} \oplus \cS^1 \cap \cS_{\sigma,\tau}$. \smallskip

We have $T(\cS^1 \cap \cS_{\sigma,\tau}) \subset \cS^1 \cap \cS_{-\sigma,-\tau}$. If $\phi \in \cS^1 \cap \cS_{\sigma,\tau}$ and $T(\phi) = 0$, then $\phi \in \cS^0 \cap \cS^1 = \{0\}$. If $\phi \in \cS^1 \cap \cS_{-\sigma,-\tau}$, then $\phi = T(\psi)$ with $\psi = - \frac 13 T(\phi) \in \cS^1 \cap \cS_{\sigma,\tau}$. This proves that $T$ is bijective.
\hfill \qed \medskip

Figure~\ref{F-hsymS0} displays the nodal and anti-nodal lines common to all functions in $H^1(\cH) \cap \cS^0_{\sigma , \tau}$, with $\sigma, \tau \in \{+ , - \}$.\medskip

\begin{figure}[htb!]
  \centering
  \includegraphics[scale=0.5]{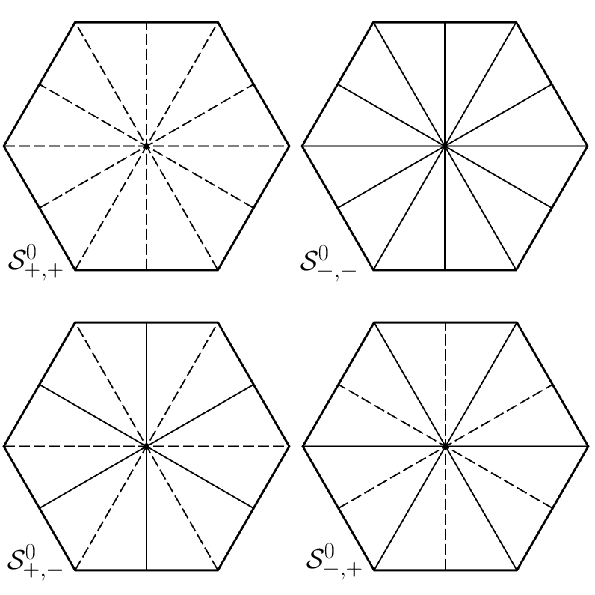}
  \caption{The spaces $\mathcal{S}^0_{\sigma,\tau}$ for
  $\sigma, \tau \in \{+,-\}$}\label{F-hsymS0}
\end{figure}

The Laplacian $\Delta$ commutes with isometries. It follows that the eigen\-spa\-ces of the Laplacian $\Delta$ in $\cH$, with either the Neumann or Dirichlet boundary condition on $\partial \cH$, decompose orthogonally according to the spaces $\cS_{\sigma , \tau}$, $\cS^0$ and $\cS^1$. More precisely, if $\cE(\lambda)$ is the eigenspace of $-\Delta$ for the eigenvalue $\lambda$ in the Neumann (resp. Dirichlet) spectrum of $\Delta$, then
\begin{equation}\label{E-hsym-24}
\cE(\lambda) = \pbigoplus_{\sigma,\tau \in \{+ \,, -\}} \big( \cE(\lambda) \cap \cS^0_{\sigma , \tau}\big) \poplus \big( \cE(\lambda) \cap \cS^1_{\sigma , \tau}\big).
\end{equation}
\medskip

\begin{remark}\label{R-hsym-2}
If $\cE(\lambda) \cap \cS^1_{\sigma , \tau} $ has dimension $p$, then by Lemma~\ref{L-hsym-6},  $\cE(\lambda) \cap \cS^1_{-\sigma , -\tau}$ has dimension $p$. It follows that $\cE(\lambda)$ has dimension at least $2p$.
\end{remark}

\begin{remark}\label{R-hsym-4}
Let $\lambda$ be a simple eigenvalue. Then, any associated eigenfunction $\phi$ is either invariant or anti-invariant under any mirror symmetry $L$ which leaves $\cH$ invariant, and invariant under $R^{*}$. It follows that $\phi \in \cS^{0}_{\sigma,\tau}$ for some pair $(\sigma,\tau)$.
\end{remark}

\begin{remark}\label{R-hsym-6}
Assume that $\phi \in \cE(\lambda) \cap \cS^0_{\sigma,\tau}$. Then,
by Courant's theorem, we have $6 \le \beta_0(\phi) \le \kappa(\lambda)$ if $(\sigma,\tau) = (+,-)$ or $(-,+)$, and $12 \le \beta_0(\phi) \le \kappa(\lambda)$ if $(\sigma,\tau) = (-,-)$. If $\phi \in \cS^{0}_{+,+}$, then $\phi$ arises from an eigenfunction of $\cT_{h}$ with Neumann boundary condition on the sides $1$ and $2$.
\end{remark}%

\subsection{Symmetries and boundary conditions on sub-domains}\label{SS-hsym-1}

Let $\cQ$ (resp. $\cP$) denote the interior of the quadrilateral (resp. the pentagon) which appears in Figure~\ref{F-hsymQP}. Let $\cR$ (resp. $\cT_h$) denote the interior of the quadrilateral (resp. of the hemiequilateral triangle) which appears in Figure~\ref{F-hsymRTh}. Then, $\overline{\cQ}$ (resp. $\overline{\cP}$) is a fundamental domain of the action of the mirror symmetry $D_1$ (resp. $M_2$), and $\overline{\cR}$ is a fundamental domain for the action of the group generated by $D_1$ and $M_2$. \medskip

Using the notation of Subsection~\ref{SS-rhome-p}, we consider the following mixed eigenvalue problems in the domains $\cH, \cP, \cQ$ and $\cR$. \medskip

\noib~ For the hexagon $\cH$, we do not decompose the boundary,
\begin{equation}\label{E-hsym-20H}
\partial \cH = \Gamma_{\cH,1}\,,
\end{equation}
and we consider  the eigenvalue problem $(\cH,\mf{b})$ with $\mf{b} \in \{\mf{n,d}\}$. \medskip

\noib~ For the quadrilateral $\cQ$, we decompose the boundary as
\begin{equation}\label{E-hsym-20Q}
\left\{
\begin{array}{ll}
\partial \cQ &= \overline{\Gamma_{\cQ,1} \sqcup \Gamma_{\cQ,2}}\,, \text{~with~} \\[5pt]
\Gamma_{\cQ,1} &= \overline{\cQ} \cap D_1 \,, \\[5pt]
\Gamma_{\cQ,2} &= \overline{\cQ} \cap \partial \cH \,,
\end{array}
\right.
\end{equation}
and we consider the  eigenvalue problems $(\cQ, \mf{ab})$, with $\mf{a, b} \in \{\mf{n,d}\}$.\medskip

\noib~ For the pentagon $\cP$, we decompose the boundary as
\begin{equation}\label{E-hsym-20P}
\left\{
\begin{array}{ll}
\partial \cP &= \overline{\Gamma_{\cP,1} \sqcup \Gamma_{\cP,2}}\,, \text{~with~} \\[5pt]
\Gamma_{\cP,1} &= \overline{\cP} \cap M_2 \,, \\[5pt]
\Gamma_{\cP,2} &= \overline{\cP} \cap \partial \cH \,,
\end{array}
\right.
\end{equation}
and we consider the  eigenvalue problems $(\cP, \mf{ab})$, with $\mf{a, b} \in \{\mf{n,d}\}$. \medskip

\noib~ For the quadrilateral $\cR$, we decompose the boundary as
\begin{equation}\label{E-hsym-20R}
\left\{
\begin{array}{ll}
\partial \cR &= \overline{\Gamma_{\cR,1} \sqcup \Gamma_{\cR,2} \sqcup \Gamma_{\cR,3}}\,, \text{~with~} \\[5pt]
\Gamma_{\cR,1} &= \overline{\cR} \cap M_2 \,, \\[5pt]
\Gamma_{\cR,2} &= \overline{\cR} \cap D_1 \,, \\[5pt]
\Gamma_{\cR,3} &= \overline{\cR} \cap \partial \cH \,,
\end{array}
\right.
\end{equation}
and we consider the  eigenvalue problems $(\cR, \mf{abc})$, with $\mf{a, b, c} \in \{\mf{n,d}\}$. \medskip

\noib~ We also consider the hemiequilateral triangle $\cT_h$, its sides ordered in decreasing order of length,  and the eigenvalue problems $(\cT_h, \mf{abc})$, with $\mf{a, b, c} \in \{\mf{n,d}\}$. For the equilateral triangle $\cT_e$, up to isometry, it is not necessary to order the sides, and we consider the  eigenvalue problems $(\cT_e, \mf{abc})$ with $\mf{a, b, c} \in \{\mf{n,d}\}$.\medskip

The boundary decompositions for the domains $\cP, \cQ, \cR$, and for the hemiequilateral triangle $\cT_h$, are illustrated in Figures~\ref{F-hsymQP} and \ref{F-hsymRTh}.\medskip

\begin{figure}[htb!]
  \centering
  \includegraphics[scale=0.25]{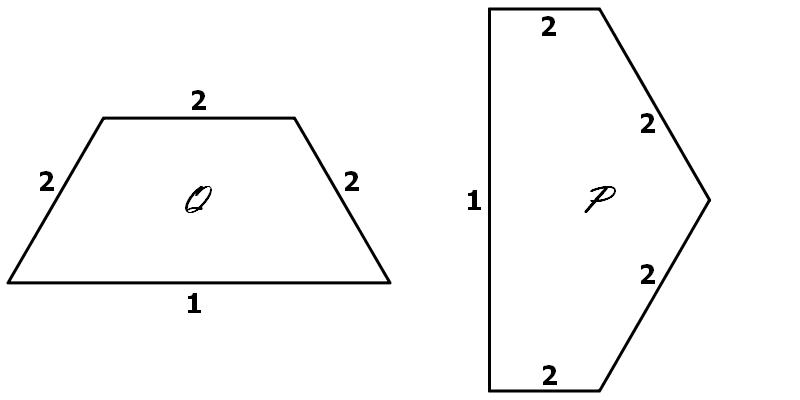}
  \caption{The sub-domains $\cQ$ and $\cP$}\label{F-hsymQP}
\end{figure}

\begin{figure}[htb!]
  \centering
  \includegraphics[scale=0.25]{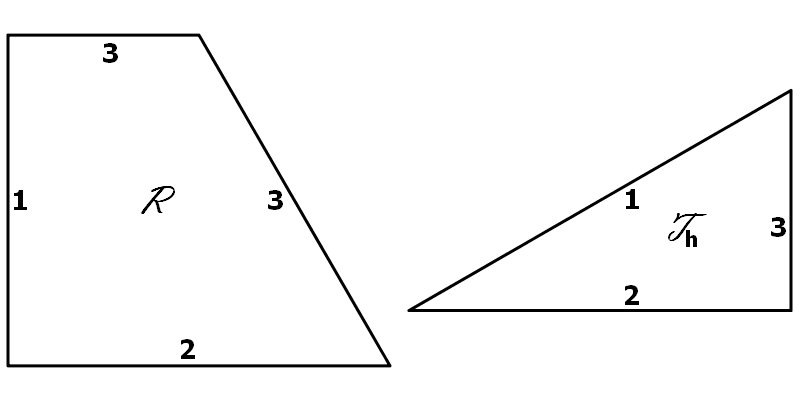}
  \caption{The sub-domains $\cR$ and $\cT_h$}\label{F-hsymRTh}
\end{figure}

Consider the eigenvalue problem $(\cH,\mf{c})$ for the hexagon, with $\mf{c} \in \{\mf{n,d}\}$. Let $\cE(\mu,\mf{c})$ be an eigenspace of $-\Delta$ for $(\cH,\mf{c})$. If $\phi \in \cE(\mu,\mf{c}) \cap \cS_{\sigma,\tau}$, then the restriction $\phi |\cR$, of the function $\phi$ to the domain $\cR$, is an eigenfunction of $- \Delta$ in $(\cR, \varepsilon(\sigma)\varepsilon(\tau)\mf{c})$, where
\begin{equation}\label{E-hsym-20s}
\varepsilon(+) = \mf{n} \text{~and~} \varepsilon(-) = \mf{d} \,,
\end{equation}
associated with the same eigenvalue $\mu$.\medskip

Conversely, let $\psi$ be an eigenfunction of $-\Delta$ in $(\cR,\mf{abc})$, associated with the eigenvalue $\mu$, where $\mf{c}$ is the given boundary condition on $\partial \cH$, and $\mf{a,b} \in \{\mf{d,n}\}$ are boundary conditions on the sides $M_2$, $D_1$. Extend $\psi$ to a function $\check{\psi}$ defined on $\cH$, by symmetry (resp. anti-symmetry) with respect to $M_2$, if $\mf{a}=\mf{n}$ (resp. if $\mf{a}=\mf{d}$), and by symmetry (resp. anti-symmetry) with respect to $D_1$, if $\mf{b}=\mf{n}$ (resp. if $\mf{b}=\mf{d}$). Then, the function $\check{\psi}$ is an eigenfunction of $-\Delta$ for $(\cH,\mf{c})$, associated with the eigenvalue $\mu$, and belongs to $\cS_{\sigma,\tau}$ with $\varepsilon(\sigma)=\mf{a}$ and $\varepsilon(\tau)=\mf{b}$. \medskip

As in Subsection~\ref{SS-rhome-p}, we have,

\begin{proposition}\label{P-hsym-2}
The eigenvalues and eigenfunctions of $(\cH,\mf{c})$ in $\cS_{\sigma,\tau}$ are in bijection with the eigenvalues and eigenfunctions of $(\cR,\mf{abc})$, with the boundary condition $\mf{a}$ on $M_2$, with $\mf{a}=\mf{d}$, if $\sigma=-$, and $\mf{a}=\mf{n}$, if $\sigma=+$; and, similarly, with the boundary condition $\mf{b}$ on $D_1$, with $\mf{b}=\mf{d}$, if $\tau=-$, and $\mf{b}=\mf{n}$, if $\tau=+$. Similar statements hold for $\cP$, $\cQ$ and $\cT_h$ respectively.
\end{proposition}%

\subsection{Identification of the first Dirichlet eigenvalues of the regular hexagon}\label{SS-hDir}


Throughout this section, we fix the Dirichlet boundary condition $\mf{d}$ on $\partial \cH$, and we denote the Dirichlet eigenvalues of $\cH$ by
\begin{equation}\label{E-hDir-2}
\delta_1(\cH) < \delta_2(\cH) \le \delta_3(\cH)\le \cdots \le \delta_6(\cH) \le \delta_7(\cH) \le \cdots \,,
\end{equation}
and the Dirichlet spectrum of the hexagon by $\spc(\cH,\mf{d})$.\\

\subsubsection{Numerical computations}\label{SSS-DirNum}

Numerical approximations for the Di\-ri\-chlet eigenvalues of the regular hexagon have been obtained by several authors, see for example \cite{BauRei1978, Jon1993, CurKut1999}, or the recent paper \cite{Jon2016}.
\medskip

The main idea, in order to make the identification of multiple Dirichlet eigenvalues of $\cH$ easier, is to take the symmetries of $\cH$ (see Section~\ref{SS-hsym-1}) into account from the start. For this purpose, one computes the eigenvalues of the domains $\cR$ and $\cT_h$, for mixed boundary conditions $\mf{abd}$, with $\mf{a,b} \in \{\mf{d,n}\}$. \medskip

Table~\ref{T-DirNum-0b} displays the first four eigenvalues of $(\cR,\mf{abd})$, as computed with \textsc{matlab}, and contains some useful relations between these eigenvalues.

\begin{table}[!htb]
\caption{$\cR$-shape, mixed boundary conditions, first four approximate eigenvalues}\label{T-DirNum-0b}
\centering
\begin{tabular}[c]{|c|c|c|c|c|c|c|c|}%
\hline
$(\cR,\mf{abd})$ & $\mu_1$ & $<$ & $\mu_2$ & $\le$ & $\mu_3$& $\le$ & $\mu_4$\\
\hline
$\mf{nnd}$ & $\phantom{1}7.16$ & $<$ & $32.45$ & $\le$ & $37.49$ & $\le$ & $\phantom{1}70.14$\\
\hline
    & \rotatebox{90}{$\cfr >$} && \rotatebox{90}{$\cfr >$} && \rotatebox{90}{$\cfr >$} && \rotatebox{90}{$\cfr >$}\\
\hline
$\mf{dnd}$ & $18.13$ & $<$ & $47.63$ & $\le$ & $60.11$ & $\le$ & $\phantom{1}94.33$\\
\hline
& ? && ? && ? && ? \\
\hline
$\mf{ndd}$ & $18.13$ & $<$ & $52.64$ & $\le$ & $60.11$ & $\le$ & $\phantom{1}94.33$\\
\hline
    & \rotatebox{90}{$\cfr >$} && \rotatebox{90}{$\cfr >$} && \rotatebox{90}{$\cfr >$} && \rotatebox{90}{$\cfr >$}\\
\hline
$\mf{ddd}$ & $32.45$ & $<$  & $70.14$ & $\le$ & $87.53$ & $\le$ & $122.82$\\
\hline
\end{tabular}%
\end{table}%

\begin{remark}\label{R-DirNum-2}
The eigenvalues in Table~\ref{T-DirNum-0b} are partially ordered `vertically'. Indeed, for $i \ge 1$, we have the strict inequalities,
\begin{equation}\label{E-DirNum-R2}
\left\{
\begin{array}{l}
\mu_i(\cR,\mf{nnd}) < \mu_i(\cR,\mf{dnd}) < \mu_i(\cR,\mf{ddd}) \,,\\[5pt]
\mu_i(\cR,\mf{nnd}) < \mu_i(\cR,\mf{ndd}) < \mu_i(\cR,\mf{ddd})\,,
\end{array}%
\right.
\end{equation}
which follow from  Proposition~\ref{P-use-2}, see \cite[Proposition~2.3]{LoRo2017}. These inequalities are indicated in the table by the (rotated) strict inequality signs. Note that it is in general not possible to compare the eigenvalues $\mu_i(\cR,\mf{dnd})$ and $\mu_i(\cR,\mf{ndd})$. This is indicated in the table by the black question marks.
\end{remark}%

Table~\ref{T-DirNum-0h} displays some eigenvalues of $(\cT_h,\mf{abd})$, for $\mf{a,b} \in \{\mf{d,n}\}$. The lower bound in the second line follows from Dirichlet monotonicity (see Subsection~\ref{SS-DirLU}). In the third line, we have used the fact due to P\'{o}lya (see \cite{LaSi2017}) that the first Dirichlet eigenvalue of a kite-shape is bounded from below by the first Dirichlet eigenvalue of a square with the same area. In the last two lines, the eigenvalues are known explicitly.

\begin{table}[!htb] 
\caption{Some eigenvalues of the hemiequilateral triangle}\label{T-DirNum-0h}
\centering
\begin{tabular}[c]{|c|c|c|c|}%
\hline
$\cS$ & $(\cT_h,\mf{abd})$ & Eigenvalue & Value $\vsp$\\[5pt]
\hline
$\cS^0_{+,+}$ & $(\cT_h,\mf{nnd})$ & $\mu_1$ & $\mu_1 \approx 7.16 \vsp$\\[5pt]
\hline
$\cS^0_{+,+}$ & $(\cT_h,\mf{nnd})$ & $\mu_2$ & $\mu_2 \approx 37.49  > 26.37 \vsp$\\[5pt]
\hline
$\cS^0_{+,-}$ & $(\cT_h,\mf{ndd})$ & $\mu_1$ & $\mu_1 \ge \frac{4\pi^2}{\sqrt{3}} \approx 22.79 \vsp$\\[5pt]
\hline
$\cS^0_{-,+}$ & $(\cT_h,\mf{dnd})$ & $\mu_1$ & $\mu_1 = 3\, \frac{16\pi^2}{9} \approx 52.64 \vsp$\\[5pt]
\hline
$\cS^0_{-,-}$ & $(\cT_h,\mf{ddd})$ & $\mu_1$ & $\mu_1 =7\, \frac{16\pi^2}{3} \approx 122.82 \vsp$\\[5pt]
\hline
\end{tabular}%
\end{table}%
\vspace{-2mm}

\begin{remark}\label{R-DirNum-2a}
The figures in Table~\ref{T-DirNum-0b} suggest that the Dirichlet ei\-gen\-values of $\cH$ come into four well separated sets $\{\delta_1(\cH)\}$, $\{\delta_2(\cH),\delta_3(\cH)\}$, $\{\delta_4(\cH),\delta_5(\cH)\}$ and $\{\delta_6(\cH)\}$.
\end{remark}%

\subsubsection{Lower and upper bounds for the Dirichlet eigenvalues}\label{SS-DirLU}

The hexa\-gon $\cH$ is inscribed in the unit disk $\cD$, and contains the disk with radius $\frac{\sqrt{3}}{2}$. By domain monotonicity for the Dirichlet eigenvalues, we have the following lower and upper bounds for the Dirichlet eigenvalues of $\cH$,
\begin{equation}\label{E-DirLU-2a}
\delta_j(\cD) < \delta_j(\cH) < \frac{4}{3}\, \delta_j(\cD) \text{~for any~} j \ge 1\,.
\end{equation}

The Dirichlet eigenvalues of the unit disk $\cD$ satisfy the relations
\begin{equation}\label{E-DirLU-4}
\left\{
\begin{array}{ll}
j_{0,1}^2 = \delta_1(\cD) &< j_{1,1}^2 = \delta_2(\cD) = \delta_3(\cD) \\
&< j_{2,1}^2 = \delta_4(\cD) = \delta_5(\cD) < j_{0,2}^2 = \delta_6(\cD)\\
&< j_{3,1}^2 = \delta_7(\cD) = \delta_8(\cD) < \cdots
\end{array}%
\right.
\end{equation}
where $j_{m,n}$ is the $n$-th positive zero of the Bessel function $J_m\,$.\medskip

Corresponding eigenfunctions are given by
\begin{equation}\label{E-DirLU-6}
\left\{
\begin{array}{lll}
\delta_1(\cD) & \leftrightsquigarrow & J_0(j_{0,1}r)\,,\\[5pt]
\delta_2(\cD) & \leftrightsquigarrow & J_1(j_{1,1}r)\cos(\theta) \text{~and~} J_1(j_{1,1}r)\sin(\theta) \,,\\[5pt]
\delta_4(\cD) & \leftrightsquigarrow & J_2(j_{2,1}r)\cos(2\theta) \text{~and~} J_2(j_{2,1}r)\sin(2\theta) \,,\\[5pt]
\delta_6(\cD) & \leftrightsquigarrow & J_0(j_{0,2}r)\,,\\[5pt]
\delta_7(\cD) & \leftrightsquigarrow & J_3(j_{3,1}r)\cos(3\theta) \text{~and~} J_3(j_{3,1}r)\sin(3\theta) \,.
\end{array}%
\right.
\end{equation}
with the nodal patterns represented in Figure~\ref{F-hdiskd6}.\medskip

\begin{figure}[htb!]
  \centering
\includegraphics[scale=0.15]{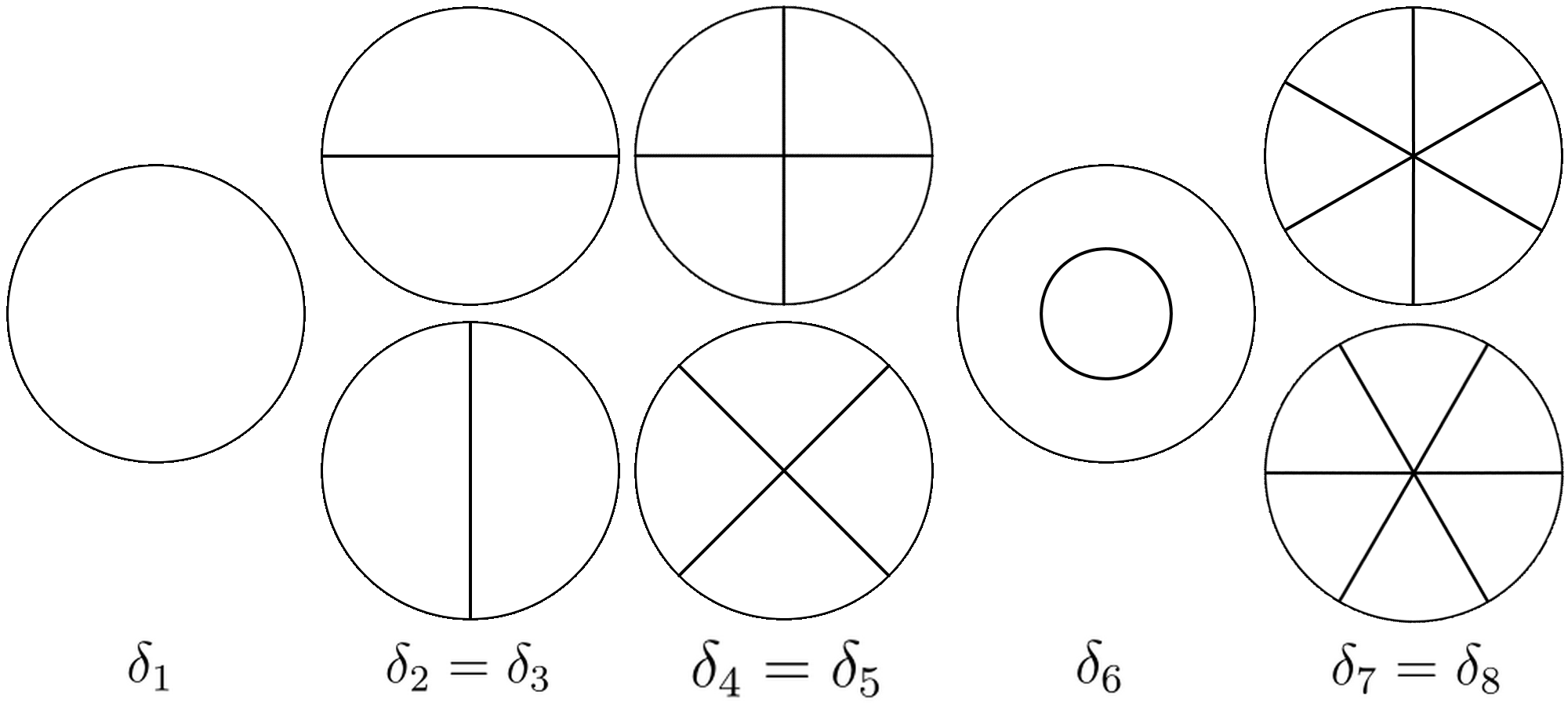}
  \caption{Nodal patterns in the first five Dirichlet eigenspaces of the unit disk}\label{F-hdiskd6}
\end{figure}

The lower and upper bounds \eqref{E-DirLU-2a} for the first eight eigenvalues are summarized in Table~\ref{T-DirLU-2}.

\begin{table}[!htb]
\caption{Bounds for the first eight Dirichlet eigenvalues of the hexagon, using domain monotonicity}\label{T-DirLU-2}
\centering
\begin{tabular}[c]{|c|c|c|}%
\hline
Eigenvalue & Lower bound & Upper bound\\
\hline
$\delta_1(\cH)$ & $\phantom{1}5.78$ & $\phantom{1}7.72$ \\
\hline
$\delta_2(\cH)$, $\delta_3(\cH)$ & $14.68$ & $19.58$ \\
\hline
$\delta_4(\cH)$, $\delta_5(\cH)$ & $26.37$ & $35.17$ \\
\hline
$\delta_6(\cH)$ & $30.47$ & $40.63$ \\
\hline
$\delta_7(\cH), \delta_8(\cH)$ & $40.70$ & $54.28$ \\
\hline
\end{tabular}%
\end{table}%

Similar bounds can be given for the first Dirichlet eigenvalues of the domains $\cP$, $\cQ$ and $\cR$, see Table~\ref{T-DirLU-4}.

\begin{table}[!htb]
\caption{Bounds for the first Dirichlet eigenvalues of $\cP$, $\cQ$ and $\cR$, using domain monotonicity}\label{T-DirLU-4}
\centering
\begin{tabular}[c]{|c|c|c|}%
\hline
Eigenvalue & Lower bound & Upper bound\\
\hline
$\delta_1(\cQ)$, $\delta_1(\cP)$  & $14.68$ & $19.58$ \\
\hline
$\delta_1(\cR)$ & $26.37$ & $35.17$ \\
\hline
\end{tabular}%
\end{table}%

\FloatBarrier

It is easy to compute the eigenvalues of a sector of the unit disk, with Neumann boundary condition on the sides of the sector, and Dirichlet boundary condition on the arc of circle. In particular, the first (resp. second) eigenvalue of such a mixed Neumann-Dirichlet problem in the circular sector of angle $\frac{\pi}{6}$ is $j_{0,1}^2$ (resp. $j_{0,2}^2$). From  domain monotonicity,  we can compare the eigenvalues of $(\cT_h,\mf{nnd})$ with the eigenvalues of the sectors with angle $\frac{\pi}{6}$, and respective radii $\frac{\sqrt{3}}{2}$ and $1$, with the Neumann boundary condition on the boundary radii, and with the Dirichlet boundary condition on the arc of circle, see Figure~\ref{F-sector}. We obtain the inequalities
\begin{equation}\label{E-DirLU-Th}
\left\{
\begin{array}{c}
~~5.78 < j_{0,1}^2 < \mu_1(\cT_h,\mf{nnd}) \approx 7.16 < \frac 43 j_{0,1}^2 < ~~7.72 \,,\\[5pt]
30.47 < j_{0,2}^2 < \mu_2(\cT_h,\mf{nnd}) \approx 37.49 < \frac 43 j_{0,2}^2 < 40.63
\,.
\end{array}
\right.
\end{equation}

\begin{figure}[htb!]
  \centering
  \includegraphics[scale=0.3]{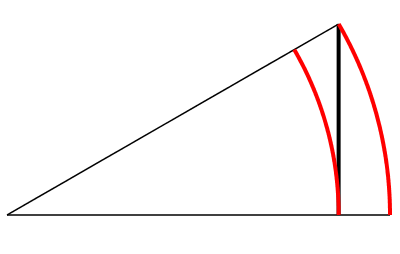}%
  \caption{Domain monotonicity}\label{F-sector}
\end{figure}

Taking into account the bounds given in Table~\ref{T-DirLU-2}, we have the relations,

\begin{equation}\label{E-DirLU-T2a}
\left\{
\begin{array}{rl}
]5.78 \,, 7.72[ \, \cap \, \sigma(\cH,d) &= \{\delta_1(\cH)\}\,, \\[5pt]
]14.68 \,, 19.58[ \, \cap \, \sigma(\cH,d) &= \{\delta_2(\cH),\delta_3(\cH)\} \,, \\[5pt]
]26.37 \,, 40.63[ \, \cap \, \sigma(\cH,d) &= \{\delta_4(\cH), \delta_5(\cH), \delta_6(\cH)\} \,, \\[5pt]
40.70 & \le \delta_7(\cH)\,.
\end{array}%
\right.
\end{equation}
Subsection~\ref{SS-hsym-1}, the bounds provided by Table~\ref{T-DirLU-4}, and inequalities \eqref{E-DirLU-Th}, imply that
\begin{equation}\label{E-DirLU-T2aa}
\left\{
\begin{array}{rl}
\mu_1(\cT_h,\mf{nnd}) & = \delta_1(\cH)\,,\\[5pt]
\{\delta_1(\cP),\delta_1(\cQ) \}  & \subset  \{\delta_2(\cH),\delta_3(\cH)\}\,,\\[5pt]
\{\delta_1(\cR),\mu_2(\cT_h,\mf{nnd})\} & \subset  \{\delta_4(\cH), \delta_5(\cH), \delta_6(\cH)\} \,.
\end{array}%
\right.
\end{equation}

We have the following proposition.

\begin{proposition}\label{P-hDir-2}
The first eigenvalues of $(\cH,\mf{d})$, satisfy the inequalities,
\begin{equation}\label{E-hDir-P2a}
\delta_1(\cH) < \delta_2(\cH) = \delta_3(\cH) < \delta_4(\cH) \le \delta_5(\cH) \le \delta_6(\cH) < \delta_7(\cH)\,.
\end{equation}
More precisely,
\begin{enumerate}
  \item A first eigenfunction $u_1$ of $(\cH,\mf{d})$ arises from a first eigenfunction of $(\cR,\mf{nnd})$. It also arises from a first eigenfunction of $(\cT_h,\mf{nnd})$.
  \item The eigenspace $\cE(\delta_2)$ has dimension $2$. It is generated by an eigenfunction $u_2$ arising from a first eigenfunction of $(\cP,\mf{d})$, and by an eigenfunction $u_3$ arising from a first eigenfunction of $(\cQ,\mf{d})$. These eigenfunctions also arise from first eigenfunctions of $(\cR,\mf{dnd})$ and $(\cR,\mf{ndd})$ respectively.
  \item The sum $\cE(\delta_4(\cH))\oplus \cE(\delta_5(\cH))\oplus \cE(\delta_6(\cH))$ has dimension $3$. It is generated by eigenfunctions $\{u,v,w\}$, where $u$ arises from a first eigenfunction of $(\cR,\mf{ddd})$, $v = T(u)$, and $w$ arises from a second eigenfunction of $(\cT_h,\mf{nnd})$. The nodal set of $w$ is a closed simple curve around the center of the hexagon.
\end{enumerate}
\end{proposition}%

\proof  We use the ideas of Subsection~\ref{SS-refpr}.

\emph{Assertion~1.~} The first Dirichlet eigenvalue is simple, and an associated eigenfunction $u_1$ does not change sign. A first eigenfunction must be invariant under all the symmetries $D_i, M_j$. This implies that $u_1$ arises from a first eigenfunction of $(\cR,\mf{nnd})$, and from a first eigenfunction of $(\cT_h,\mf{nnd})$.

\emph{Assertion~2.~} Let $\psi$ be a first eigenfunction of $(\cQ,\mf{d})$. It does not change sign in $\cQ$, and must be invariant with respect to $M_2$. This means that it arises from a first eigenfunction of $(\cR,\mf{ndd})$. Extend $\psi$ to $u_3$ on $\cH$, so that it is anti-invariant under $D_1$. The function  $u_3$ is an eigenfunction of $(\cH,\mf{d})$. It is associated with $\delta_1(\cQ)$, belongs to $\cS_{- , +}$, and its nodal set is $D_1 \cap \cH$, so that $u_3 \not \in \cS^0_{- , +}$. Similarly, let $\theta$ be a first of $(\cP,\mf{d})$. It does not vanish in $\cP$, and is invariant with respect to $D_1$. It arises from a first eigenfunction of $(\cR,\mf{dnd})$, and can be extended to $u_2$ on $\cH$, an eigenfunction of $(\cH,\mf{d})$, associated with $\delta_1(\cP)$,  belonging to $\cS_{+ , -}$, and whose nodal set is $M_2 \cap \cH$, so that $u_2 \not \in \cS^0_{+ , -}$. Applying Lemma~\ref{L-hsym-6}, and \eqref{E-DirLU-T2aa}, we conclude that we can choose $u_3 = T(u_2)$, and hence that
\begin{equation}\label{E-hDir-10}
\delta_2(\cH) = \delta_3(\cH) = \delta_1(\cP) = \delta_1(\cQ).
\end{equation}

\emph{Assertion~3.~} We reason as in the proof of Assertion~2. From a first eigenfunction $\phi$ of $(\cR,\mf{d})$, we obtain an eigenfunction $u$ of $(\cH,\mf{d})$, associated with $\delta_1(\cR)$, belonging to $\cS_{-,-}\,$, whose nodal set is $(D_1 \cup M_2) \cap \cH\,$. Then $u$  does not belong to $\cS^{0}_{-,-}\,$.  Applying Lemma~\ref{L-hsym-6}, and \eqref{E-DirLU-T2aa}, we can choose $v=T(u)$. Using \eqref{E-DirLU-T2aa}, more precisely the fact that $\mu_2(\cT_h,\mf{nnd}) \in \spc(\cH,\mf{d})$, we now choose $w$ to arise from a second eigenfunction $\xi$ of $(\cT_h,\mf{nnd})$.\medskip

Because $w$ is a Dirichlet eigenfunction of the convex set $\cH$, the nodal set of $w$ has the properties described in \cite{Al1994}. The function $\xi$ has two nodal domains, and its nodal set must be a single simple line which is either closed inside $\cT_h$, or goes from one side to another side (including the possibility to start or arrive at a vertex). Looking at all the possible configurations, we see that the function $w$ would have at least seven nodal domains (this is prohibited by Courant's theorem), except in one case, when the nodal set of $\xi$ is a curve from the open side of $\cT_h$ labelled $1$, to the open side labelled $2$. In this case, the function $w$ has a closed nodal line and two nodal domains.\medskip

Note: We know that $\dim \left( \cE(\delta_4) \oplus \cE(\delta_5) \oplus \cE(\delta_6) \right) = 3$. According to Remark~\ref{R-hsym-2}, this implies that
$\left( \cE(\delta_4) \oplus \cE(\delta_5) \oplus \cE(\delta_6) \right) \cap \cS^0 \not = \{0\}$.\medskip

The proposition is proved. \hfill \qed \medskip

\begin{remark}\label{R-hex-dir1}
We can determine which eigenvalues among the first four eigenvalues of $(\cR,\mf{abd})$, $\mf{a,b} \in \{\mf{d,n}\}$, might possibly be $\delta_4(\cH)$. Table~\ref{T-DirNum-0bp} takes Remark~\ref{R-DirNum-2} and Assertions~1 and 2 into account.  The word ``no''  in a cell means that the corresponding eigenvalue $\mu_i(\cR,\mf{abd})$ cannot be equal to $\delta_4(\cH)$ due to the known inequalities on these eigenvalues. The only remaining possibilities are $\delta_4(\cH)  = \mu_2(\cR,\mf{nnd})$ (which might be a multiple eigenvalue), and $\delta_4(\cH) = \mu_1(\cR,\mf{ddd})$.
\end{remark}%

\begin{table}[!htb]
\caption{Possible choices for $\delta_4(\cH)$}\label{T-DirNum-0bp}
\centering
\begin{tabular}[c]{|c|c|c|c|c|c|c|c|}%
\hline
$(\cR,\mf{abd})$ & $\mu_1$ & $<$ & $\mu_2$ & $\le$ & $\mu_3$& $\le$ & $\mu_4$\\
\hline
$\mf{nnd}$ & $\delta_1(\cH)$ & $<$ &  & $\le$ &  & $\le$ &\\
\hline
    & \rotatebox{90}{$\cfr >$} && \rotatebox{90}{$\cfr >$} && \rotatebox{90}{$\cfr >$} && \rotatebox{90}{$\cfr >$}\\
\hline
$\mf{dnd}$ & $\delta_2(\cH)=\delta_3(\cH)$ & $<$ & no & $\le$ & no & $\le$ & no\\
\hline
$\mf{ndd}$ & $\delta_2(\cH)=\delta_3(\cH)$ & $<$ & no & $\le$ & no & $\le$ & no\\
\hline
    & \rotatebox{90}{$\cfr >$} && \rotatebox{90}{$\cfr >$} && \rotatebox{90}{$\cfr >$} && \rotatebox{90}{$\cfr >$}\\
\hline
$\mf{ddd}$ &  & $<$  & no & $\le$ & no & $\le$ & no\\
\hline
\end{tabular}%
\end{table}%

\subsection{Numerical results and $\ecp(\cH,\mf{d})$ }\label{SS-ch-ecp-d}

Using the numerical approximations given in Table~\ref{T-DirNum-0b}, we infer the (numerical) lower bound $ \delta_6(\cH) > 35.17 $. This implies that $\delta_6(\cH)$ is simple. It follows that $u_{6}$ arises from the second eigenfunction of $\cT_h$, with mixed boundary condition $\mf{nnd}$ (Dirichlet on the smaller side of $\cT_h$, Neumann on the other sides). This provides the following numerical extension of Proposition~\ref{P-hDir-2},

\begin{state}\label{P-hDir-2e}
The Dirichlet eigenvalues of $\cH$ satisfy,
\begin{equation}\label{E-hDir-P2b}
\delta_1(\cH) < \delta_2(\cH) = \delta_3(\cH) < \delta_4(\cH) = \delta_5(\cH) < \delta_6(\cH) < \delta_7(\cH)\,,
\end{equation}
and
\begin{equation}\label{E-hDir-P2c}
\delta_4(\cH) = \delta_5(\cH) = \delta_1(\cR)\,,
\end{equation}
The eigenspace $\cE(\delta_4)$ has dimension $2$, and is generated by an eigenfunction $u_4$ which arises from the first eigenfunction of $(\cR,\mf{ddd})$ and the function $u_5=T(u_4)$. The eigenfunction $u_6$ associated with $\delta_6(\cH)$ arises from the second eigenfunction of $(\cT_h,\mf{nnd})$, and its nodal set is a simple closed curve enclosing the center of the hexagon.
\end{state}%

Figure~\ref{F-Hx} displays the nodal patterns of first six Dirichlet eigenfunctions of $\cH$. 

\begin{figure}[!htb]
  \centering
  \includegraphics[width=7cm]{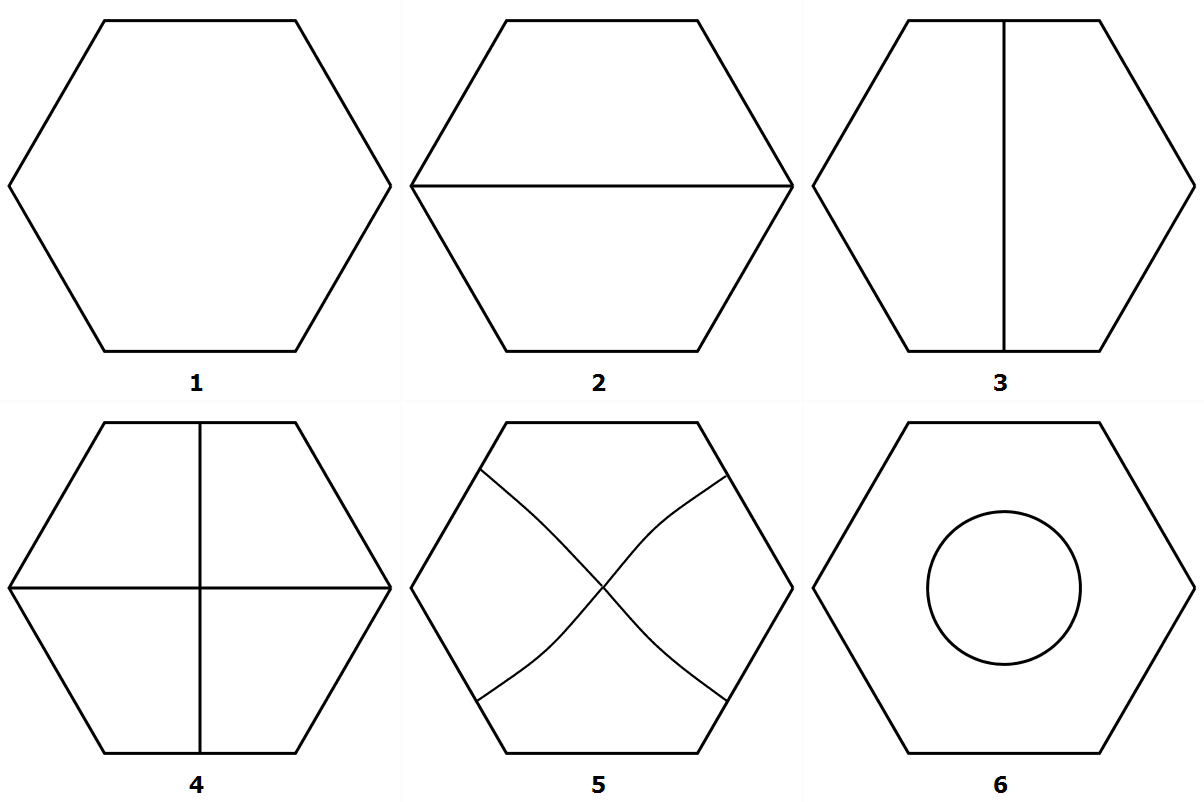}
  \caption{$(\cH,\mf{d})$: nodal structure for the first  six eigenfunctions}\label{F-Hx}
\end{figure}

\FloatBarrier

Plotting the nodal set of the linear combination $u_{6} + a \, u_{1}$ for several values of $a$, one finds some values of $a$ for which this function has $7$ nodal domains, see Figure~~\ref{F-hexD-noecp}.

\begin{figure}[!htb]
  \centering
  \includegraphics[width=8cm]{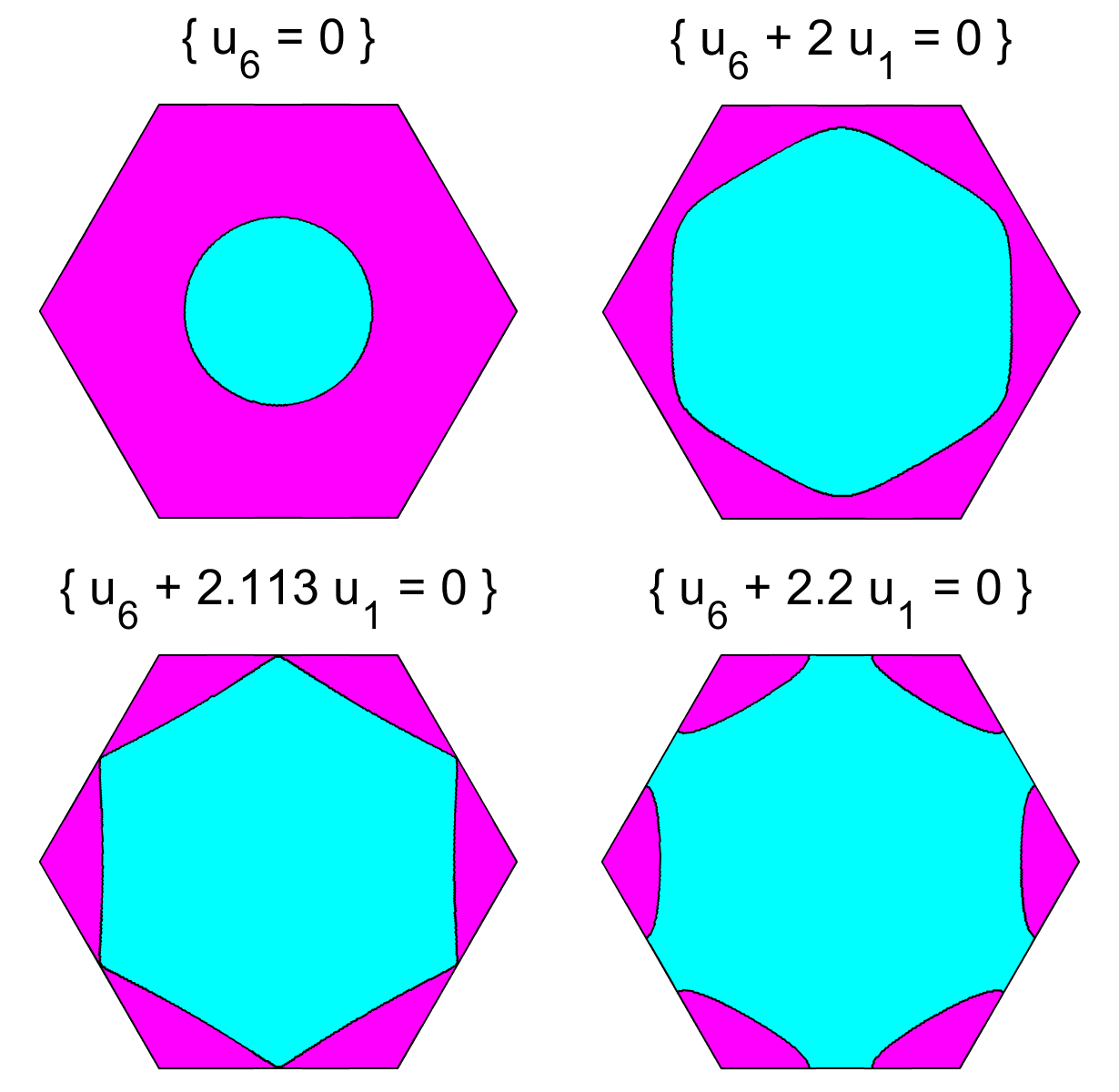}
  \caption{$(\cH,\mf{d})$: the $\ecp$ is false in $\cE(\delta_1)\oplus \cE(\delta_6)$}\label{F-hexD-noecp}
\end{figure}

\begin{state}\label{Cl-num-4d}
Figure~\ref{F-hexD-noecp} provides a numerical evidence that the $\ecp(\cH,\mf{d})$ is false.
\end{state}%


\begin{remark}\label{R-num-5dr}
For Statement~\ref{Cl-num-4d}, we do not really need to separate $\delta_6(\cH)$ from $\delta_5(\cH)$. It suffices to use Proposition~\ref{P-hDir-2}, and more precisely the fact that there exists an eigenfunction in $\cE(\delta_4) \oplus \cE(\delta_5) \oplus \cE(\delta_6)$, which arises from a second eigenfunction of $(\cT_h,\mf{nnd})$. As in Subsection~\ref{SS-rhe-n},  we then need to know the nodal patterns in $\cE\left( \mu_1(\cT_h,\mf{nnd})\right) \oplus \cE\left( \mu_2(\cT_h,\mf{nnd}) \right)$, or equivalently the nodal patterns in $\cE\left( \mu_1(\cT_e,\mf{nnd})\right) \oplus \cE\left( \mu_2(\cT_e,\mf{nnd}) \right)$, see Remark~\ref{R-rhe-ecp2s} and Figure~\ref{F-Th-nnd}.
\end{remark}%


\subsection{Identification of the first Neumann eigenvalues of the regular hexagon}\label{SS-hNeu}

\subsubsection{Numerical computations and preliminary remarks}\label{SS-hNeu-1}

We did not find numerical computations of the Neumann eigenvalues of the hexa\-gon in the literature. We use the same method as in Subsection~\ref{SS-hDir}.\medskip

Given an eigenspace $\cE(\lambda)$ of $-\Delta$ for $(\cH,\mf{n})$, we apply Lemma~\ref{L-hsym-6}, and write
\begin{equation}\label{E-hNeu-4}
\cE(\lambda) = \pbigoplus_{\sigma,\tau \in \{+,-\}} \big( \cE(\lambda) \cap \cS^{0}_{\sigma,\tau}\big) \poplus \big( \cE(\lambda) \cap \cS^{1}_{\sigma,\tau}\big).
\end{equation}

This means that to determine the eigenvalues of $(\cH,\mf{n})$, it suffices to list the eigenvalues of $(\cR,\mf{abn})$, with $\mf{a,b} \in \{\mf{d,n}\}$, and to re-order them in non-decreasing order.\medskip

Table~\ref{T-hNeu-2} displays the approximate values of the first four eigenvalues of $(\cR,\mf{abn})$, as calculated by \textsc{matlab}.

\begin{table}[!htb]
\caption{First four eigenvalues for $(\cR, \mf{abn})$, $\mf{a, b} \in \{\mf{d,n}\}$}\label{T-hNeu-2}
\centering
\begin{tabular}[c]{|c|c|c|c|c|c|c|c|}%
\hline
$(\cR,\mf{abd})$ & $\mu_1$ & $<$ & $\mu_2$ & $\le$ & $\mu_3$& $\le$ & $\mu_4$\\
\hline
$\mf{nnn}$ & $0$ & $<$ & $10.87$ & $\le$ & $17.55$ & $\le$ & $33.45$\\
\hline
    & \rotatebox{90}{$\cfr >$} && \rotatebox{90}{$\cfr >$} && \rotatebox{90}{$\cfr >$} && \rotatebox{90}{$\cfr >$}\\
\hline
$\mf{dnn}$ & $4.04$ & $<$ & $17.55$ & $\le$ & $32.91$ & $\le$ & $49.90$\\
\hline
& ? && ? && ? && ? \\
\hline
$\mf{ndn}$ & $4.04$ & $<$ & $24.90$ & $\le$ & $32.91$ & $\le$ & $49.90$\\
\hline
    & \rotatebox{90}{$\cfr >$} && \rotatebox{90}{$\cfr >$} && \rotatebox{90}{$\cfr >$} && \rotatebox{90}{$\cfr >$}\\
\hline
$\mf{ddn}$ & $10.87$ & $<$  & $33.45$ & $\le$ & $54.77$ & $\le$ & $71.71$\\
\hline
\end{tabular}%
\end{table}%

\begin{remark}\label{R-hNeu-E2}
The following inequalities follow from Proposition~\ref{P-use-2}, see \cite{LoRo2017},
\begin{equation}\label{E-hNeu-2a}
\left\{
\begin{array}{l}
\mu_i(\cR, \mf{nnn}) < \mu_i(\cR, \mf{dnn})  < \mu_i(\cR, \mf{ddn})\,, \\[5pt]
\mu_i(\cR, \mf{nnn}) < \mu_i(\cR, \mf{ndn})  < \mu_i(\cR, \mf{ddn})\,.
\end{array}%
\right.
\end{equation}
\end{remark}%
These inequalities are indicated in Table~\ref{T-hNeu-2} by the (rotated) strict inequality signs. The question marks indicate that one cannot compare the other values.\medskip

 Eigenfunctions in $\cE(\lambda) \cap \cS^{0}_{\sigma,\tau}$ correspond to eigenfunctions of $-\Delta$ for $(\cT_h,\mf{abn})$ with $\mf{a}=\mf{d}$ (resp. $\mf{a}=\mf{n}$) if $\tau = -$ (resp. $\tau = +$), and similarly for $\mf{b}$, with $\sigma$. Table~\ref{T-hNeu-6} displays the first non trivial eigenvalue of $(\cT_h,\mf{abn})$.

\begin{table}[!htb]
\caption{Least non trivial eigenvalues for the hemiequilateral triangle}\label{T-hNeu-6}
\centering
\begin{tabular}[c]{|c|c|c|c|}%
\hline
$\cS$ & $(\cT_h,\mf{abn})$ & Eigenvalue & Value $\vsp$\\[5pt]
\hline
$\cS^0_{+,+}$ & $(\cT_h,\mf{nnn})$ & $\mu_2$ & $\mu_2 = \frac{16\pi^2}{9}\approx 17.55 \vsp$\\[5pt]
\hline
$\cS^0_{+,-}$ & $(\cT_h,\mf{ndn})$ & $\mu_1$ & $\mu_1=\frac{16\pi^2}{9} \approx 17.55 \vsp$\\[5pt]
\hline
$\cS^0_{-,+}$ & $(\cT_h,\mf{dnn})$ & $\mu_1$ & $\mu_1 > \frac{4\pi^2}{\sqrt{3}} > 22.79 \vsp$\\[5pt]
\hline
$\cS^0_{-,-}$ & $(\cT_h,\mf{ddn})$ & $\mu_1$ & $\mu_1 > \frac{16\pi^2}{3} > 52.64 \vsp$\\[5pt]
\hline
\end{tabular}%
\end{table}%

\begin{remarks}\label{R-hNeu-E6}
(1)~The first eigenvalue $\mu_1(\cT_h,\mf{nnn})$ is $0$. The second eigenvalue $\mu_2(\cT_h, \mf{nnn})$ is also the second eigenvalue of an equilateral triangle with Neumann boundary condition. The corresponding eigenfunction has a nodal line which is a curve from side $1$ to side $2$ of $\cT_h$.\\
(2)~In the third line of Table~\ref{T-hNeu-6}, we use the fact that $\mu_1(\cT_h,\mf{dnn})$ is the first Dirichlet eigenvalue of an equilateral rhombus. It is bounded from below by the first Dirichlet eigenvalue of a square with the same area (P\'{o}lya, see \cite{LaSi2017}).\\
(3)~In the fourth line of Table~\ref{T-hNeu-6}, we use the fact that $\mu_1(\cT_h,\mf{ddn})$ is the first Dirichlet eigenvalue of an isosceles triangle with sides $(1,1,\sqrt{3})$. It is bounded from below by the first Dirichlet eigenvalue of the equilateral triangle with the same area (P\'{o}lya, see \cite{LaSi2017}). Note that $\mu_1(\cT_h,\mf{ddn}) > \mu_1(\cT_h,\mf{dnn})$ according to Proposition~\ref{P-use-2}.
\end{remarks}%

One can also compute the eigenvalues of $(\cH,\mf{n})$ directly, without taking the symmetries into account. The first Neumann eigenvalues of the hexagon are given in Table~\ref{T-num-2n}.

\begin{table}[!htb]
\caption{First non-trivial Neumann eigenvalues of $\cH$}\label{T-num-2n}
\centering
\begin{tabular}{|c|c|c|}
\hline
Eigenvalue of $\cH$ & Approximation & Eigenvalue of $\cR$\\[5pt]
\hline
$\nu_2(\cH)$ & $\approx 4.04$ & $\mu_1(\cR,\mf{dnn})$\\[5pt]
\hline
$\nu_3(\cH)$ & $\approx 4.04$ & $\mu_1(\cR,\mf{ndn})$\\[5pt]
\hline
$\nu_4(\cH)$ & $\approx 10.87$ & $\mu_1(\cR,\mf{ddn})$\\[5pt]
\hline
$\nu_5(\cH)$ & $\approx 10.87$ & $\mu_2(\cR,\mf{nnn})$\\[5pt]
\hline
$\nu_6(\cH)$ & $\approx 17.55$ & $\mu_2(\cR,\mf{dnn})$\\[5pt]
\hline
$\nu_7(\cH)$ & $\approx 17.55$ & $\mu_3(\cR,\mf{nnn})$\\[5pt]
\hline
$\nu_8(\cH)$ & $\approx 24.90$ & $\mu_2(\cR,\mf{ndn})$\\[5pt]
\hline
\end{tabular}
\end{table}

\FloatBarrier

Figure~\ref{F-hexN} displays the nodal patterns of eigenfunctions associated with the eigenvalues $\nu_i(\cH), 2\le i \le 7$. \medskip

\begin{figure}[!htb]
  \centering
  \includegraphics[width=11cm]{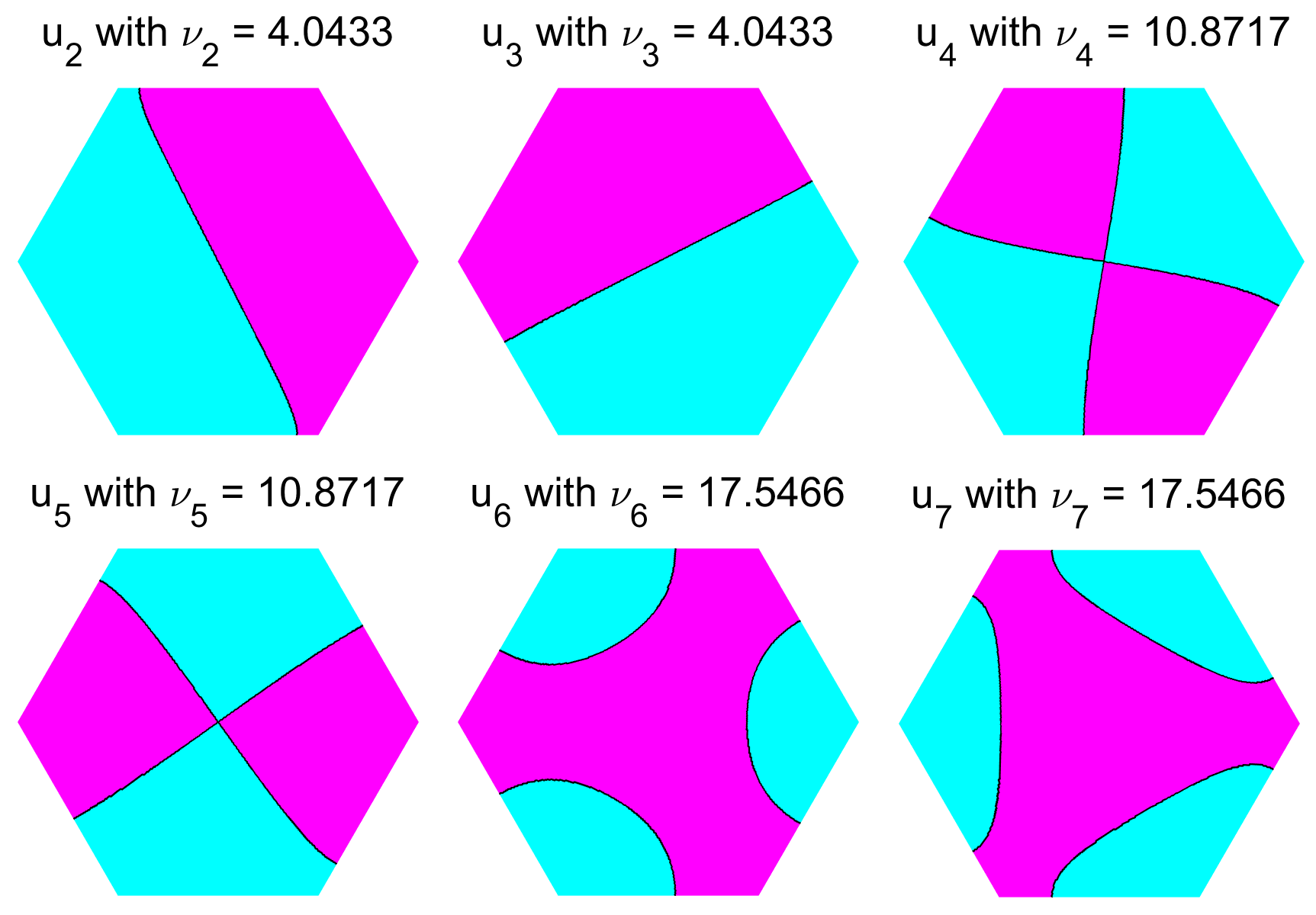}
  \caption{$(\cH,\mf{n})$: nodal patterns $u_2$ -- $u_7$}\label{F-hexN}
\end{figure}

\begin{remark}\label{R-hNeu-E6b}
The figures in Table~\ref{T-hNeu-2} suggest that the Neumann eigenvalues of the hexagon come into well separated sets:
\begin{equation*}
\left\{
\begin{array}{rl}
\nu_1(\cH) & = 0\,,\\[5pt]
\{\nu_2(\cH),\nu_3(\cH)\} &\subset\, ]3,5[\,,\\[5pt]
\{\nu_4(\cH),\nu_5(\cH)\} & \subset\, ]6,14[\,,\\[5pt]
\{\nu_6(\cH),\nu_7(\cH)\} & \subset\, ]15,20[\,,\\[5pt]
\nu_8(\cH) & > 21\,.
\end{array}%
\right.
\end{equation*}
\end{remark}

In the following subsections, we analyze the possible eigenspaces and, more precisely, the double eigenvalues. Note that for Neumann eigenvalues we do not have monotonicity inequalities as the ones we used for Dirichlet eigenvalues in Subsection~\ref{SS-DirLU}, so that we have to rely on the numerical evidence provided by Remark~\ref{R-hNeu-E6b}.

\subsubsection{Analysis of the possible eigenspaces of $(\cH,\mf{n})$}\label{SS-hNeu-2}

We divide the analysis into several steps.\medskip

\textbf{Step~1: eigenvalue $\nu_1(\cH)$}.~
The first Neumann eigenvalue is zero, and simple, with a corresponding eigenfunction $u_1$ which is constant. We have $u_1 \in \cS^0_{+,+}$, and $\nu_1(\cH) = \mu_1(\cT_h,\mf{nnn})$.\medskip

\textbf{Step~2: eigenvalue $\nu_2(\cH)$}.~
Let $\cE_2 = \cE\left( \nu_2(\cH) \right)$ be the corresponding eigen\-space. \smallskip

\noid~ We claim that
\begin{equation}\label{E-hNeu-Claim2a}
\cE_2 \cap \cS^0 = \{0\}\,.
\end{equation}

Indeed, Courant's nodal domain theorem and Lemma~\ref{L-hsym-6} imply that $\cE_2 \cap \cS^0_{\sigma,\tau} = \{0\}$ unless $(\sigma,\tau) = (+,+)$. Assume that there exists some $0 \not = \phi \in \cE_2 \cap \cS^0_{+,+}$. The restriction of $\phi$ to $\cT_h$ would be an eigenfunction of $-\Delta$ for $(\cT_h,\mf{nnn})$. Because $\nu_2(\cH)$ is the least non zero eigenvalue, we would have $\nu_2(\cH) = \mu_2(\cT_h,\mf{nnn})$, whose eigenfunction is known, with nodal set an arc from the side $1$ to the side $2$. The function $\phi$ would have a closed nodal line bounding a nodal domain strictly contained in the interior of $\cH$, and we would have $\nu_2(\cH) > \delta_1(\cH)$, contradicting the fact that $\nu_3(\cH) \le \delta_1(\cH)$ according to \cite[Theorem~4.2]{LeWe1986}.\smallskip

As a by-product of \eqref{E-hNeu-Claim2a}, Lemma~\ref{L-hsym-6} tells us that the map $T$, defined by \eqref{E-hsym-L4a}, is a bijection from $\cE_2$ to $\cE_2$.\smallskip


\noid~ We claim that
\begin{equation}\label{E-hNeu-Claim2b}
\cE_2 \cap \cS^1_{-,-} = \{0\} \text{~~and~~}
\cE_2 \cap \cS^1_{+,+} = \{0\}\,.
\end{equation}

The first assertion is clear by Courant's theorem. The second assertion follows from the fact that the map $T$ is a bijection from $\cS^1_{+,+}$ onto $\cS^1_{-,-}$ which commutes with $\Delta$.\medskip

\noid~ We claim that
\begin{equation}\label{E-hNeu-Claim2ba}
\dim \left( \cE_2 \cap \cS^1_{+,-} \right) = \dim \left( \cE_2 \cap \cS^1_{-,+}\right) = 1\,,
\end{equation}
and hence that $\mathrm{mult}\left(\nu_2(\cH)\right) = 2\,$. \smallskip

Indeed, using the map $T$ again, we see that the spaces $\cE_2 \cap \cS^1_{+,-}$ and $\cE_2 \cap \cS^1_{-,+}$ have the same dimension. According to \cite{HOMN1999}, see the statement p.~1170, line~(-8), the multiplicity of $\nu_2(\cH)$ is less than or equal to $3$, and we can conclude that this dimension must be $1$.  Here is an alternative argument for the case at hand.
It suffices to prove that the dimension of $\cE_2$ cannot be larger than or equal to $4$. Indeed, assume that $\dim \cE_2 \ge 4$. One could then find a point $x_0 \in \cH$, and an eigenfunction $u_4 \in \cE_2$ such that $u_4(x_0) \not = 0$. The subspace $\cE_{2,x_0} = \{u \in \cE_2 ~|~ u(x_0)=0\}$ would have dimension $3$, with a basis $u_1, u_2, u_3$. The three vectors $\nabla u_1(x_0),\nabla u_2(x_0),\nabla u_3(x_0) \in \R^2$ would be linearly dependent, and we would then find a nontrivial $u \in \cE_2$ such that $u(x_0) = \nabla u(x_0) = 0$. The nodal set of $u$ would contain at least four semi-arcs emanating from $x_0$, and we would reach a contradiction with the fact that $u$ has only two nodal domains by Courant's theorem.
%


\smallskip
Because $\mu_2(\cR,\mf{nnn})$ is an eigenvalue of $(\cH,\mf{n})$,  we have proved the following lemma.

\begin{lemma}\label{L-hNeu-2} The eigenvalue $\nu_2(\cH)$ has multiplicity $2$,
\begin{equation}\label{E-hNeu-22}
\nu_2(\cH) = \mu_1(\cR,\mf{dnn}) = \mu_1(\cR,\mf{ndn})\,,
\end{equation}
and corresponding eigenfunctions $u_2, u_3$ arise from the first eigenfunctions of $-\Delta$ for $(\cR,\mf{dnn})$ and $(\cR,\mf{ndn})$. Furthermore, $$\nu_2(\cH) = \nu_3(\cH) < \mu_2(\cR,\mf{nnn})\,.$$
\end{lemma}%

\textbf{Step~3: eigenvalue $\nu_4(\cH)$}.~
Let $\cE_4 = \cE\left( \nu_4(\cH) \right)$ be the ei\-gen\-space associated with the eigenvalue $\nu_2(\cH)$.\smallskip

\noid~ We claim that
\begin{equation}\label{E-hNeu-Claim3a}
\cE_4 \cap \cS^0 = \{0\}\,.
\end{equation}

Indeed, by Courant's theorem and Lemma~\ref{L-hsym-6},
\begin{equation}\label{E-hNeu-32}
\cE_4 \cap \cS^0_{\sigma,\tau} = \{0\},
\end{equation}
unless $(\sigma,\tau) = (+,+)$. Assume that there exists some $0 \not = \phi \in \cS^0_{+,+}$. Then, we would have $\nu_4(\cH) = \mu_2(\cT_h,\mf{nnn}) = \frac{16\pi^2}{9}$. Observe that $\nu_2(\cT_e) = \mu_2(\cT_h,\mf{nnn}) = \mu_1(\cT_h,\mf{ndn})$. This means that $\cE_4$ would also contain a function in $\cS^0_{+,-}$ having $6$ nodal domains which would contradict Courant's theorem. \medskip

From \eqref{E-hNeu-Claim3a} and Proposition~\ref{P-use-4} (\cite[Theorem~1.1]{Siu2016}), we deduce that
\begin{equation}\label{E-hNeu-Claim3aa}
\nu_4 < \mu_2(\cT_h,\mf{nnn}) = \mu_1(\cT_h,\mf{ndn}) < \mu_1(\cT_h,\mf{dnn}) < \mu_1(\cT_h,\mf{ddn})\,.
\end{equation}

\noid~ We claim that
\begin{equation}\label{E-hNeu-34}
\cE_4 \cap \cS^1_{+,-} = \{0\} \text{~and~} \cE_4 \cap \cS^1_{-,+} = \{0\}\,.
\end{equation}

Indeed, assume that $\cE_4 \cap \cS^1_{+,-} \not = \{0\}$ or, equivalently using the map $T$, that $\cE_4 \cap \cS^1_{-,+} \not = \{0\}$. Then, we would have
\begin{equation}\label{E-hNeu-34a}
\nu_4 = \mu_2(\cR,\mf{ndn}) = \mu_2(\cR,\mf{dnn})\,.
\end{equation}

These eigenvalues are strictly larger than $\mu_2(\cR,\mf{nnn})$ by \eqref{E-hNeu-2a}, and this would contradict the fact that $\nu_3(\cH) < \mu_2(\cR,\mf{nnn})$, see Step~2, because $\mu_2(\cR,\mf{nnn})$ is an eigenvalue for $(\cH,\mf{n})$.\smallskip

As a by-product, we have the inequalities
\begin{equation}\label{E-hNeu-34b}
\left\{
\begin{array}{l}
\nu_4 < \mu_2(\cR,\mf{dnn}) < \mu_2(\cR,\mf{ddn})\,,\\[5pt]
\nu_4 < \mu_2(\cR,\mf{ndn}) < \mu_2(\cR,\mf{ddn})\,.
\end{array}
\right.
\end{equation}

\noid~ It follows from the above arguments that we must have,
\begin{equation}\label{E-hNeu-38}
\dim\left( E_4 \cap \cS^1_{-,-}\right) = \dim\left( E_4 \cap \cS^1_{+,+}\right) > 0\,,
\end{equation}
and hence that $\dim \cE_4 \ge 2$, so that  $\nu_4(\cH) = \nu_5(\cH)=\mu_1(\cR,\mf{ddn}) = \mu_2(\cR,\mf{nnn})$, with corresponding eigenfunction $u_4, u_5$ for $(\cH,\mf{n})$.\medskip

\begin{remark}\label{R-hNeu38a}
According to Table~\ref{T-hNeu-2} and Remark~\ref{R-hNeu-E6b}, $\dim \cE_4 = 2$.
\end{remark}%

\textbf{Step~4: eigenvalue $\nu_6(\cH)$}.~
So far, we have established the following facts
%
%
\begin{equation*}\label{E-hNeu-42}
\begin{array}{c}
\nu_1(\cH) < \nu_2(\cH) = \nu_3(\cH) < \nu_4(\cH) = \nu_5(\cH)\le \cdots \\[5pt] \text{~or}\\[5pt]
\mu_1(\cR,\mf{nnn}) < \mu_1(\cR,\mf{dnn}) = \mu_1(\cR,\mf{ndn}) < \mu_2(\cR,\mf{nnn}) = \mu_1(\cR,\mf{ddn})\,.
\end{array}%
\end{equation*}

The next eigenvalue $\nu_6(\cH)$ should belong to the set,
\begin{equation}\label{E-hNeu-44}
\left\lbrace \mu_2(\cR,\mf{dnn}), \mu_2(\cR,\mf{ndn}), \mu_2(\cR,\mf{ddn}), \mu_3(\cR,\mf{nnn})\right\rbrace
\end{equation}
We can exclude $\mu_2(\cR,\mf{ddn})$ because it is larger than both $\mu_2(\cR,\mf{dnn})$ and $\mu_2(\cR,\mf{ndn})$ according to \eqref{E-hNeu-2a} (\cite[Proposition~2.3]{LoRo2017}).\medskip

%

The eigenvalues of $(\cT_h,\mf{abn})$, $\mf{a,b} \in \{\mf{d,n}\}$, are eigenvalues of $(\cH,\mf{n})$. Using Table~\ref{T-hNeu-6} and Remark~\ref{R-hNeu-E6b}, we can conclude that
%
%
%
%
\begin{equation}\label{E-hNeu-48}
\nu_6(\cH) = \nu_7(\cH) = \mu_1(\cT_h,\mf{ndn}) = \mu_2(\cT_h,\mf{nnn})\,,
\end{equation}
and that associated eigenfunctions $u_6, u_7$  arise from the first and second eigenfunctions of $(\cT_h,\mf{ndn})$.\medskip

\begin{state}\label{Cl-hNeu-ecp}
From the numerical evidence in Remark~\ref{R-hNeu-E6b}, we conclude that $\nu_i(\cH)$ have multiplicity $2$ for $i\in \{2,4,6\}$, and that $\nu_6(\cH)$ and $\nu_7(\cH)$ arise from eigenvalues of $(\cT_h,\mf{abn})$.
\end{state}%

\subsection{Numerical computations and $\ecp(\cH,\mf{n})$}\label{SS-ch-ecp-n}

The first Neumann eigenvalue of the hexagon, $\nu_1(\cH)$, is $0$, with associated eigenfunction $u_{1} \equiv 1$. As sixth Neumann eigenfunction $u_6$ of the hexagon, we can choose the function which arises from an eigenfunction for $\mu_2(\cT_h,\mf{nnn})$, or equivalently from a $D$-invariant second eigenfunction $\psi$ of $(\cT_e,\mf{n})$. It follows from \cite[Section~3]{BH2018-ecp1} that $\ecp(\cT_e,\mf{n})$ is false, i.e., that there exists some real value $a$ such that $\psi + a$ has three nodal domains in $\cT_e$. It follows that $u_6 + a$ has seven nodal domains, so that $\ecp(\cH,\mf{n})$ is false.\medskip

Alternatively, we can look at $\mu_3(\cR,\mf{nnn}) = \mu_2(\cT_h,\mf{nnn})$. Figure~\ref{F-num-8} displays the nodal pattern and the level lines of an eigenfunction for $\mu_3(\cR,\mf{nnn})$. By reflection with respect to the lines $D_1$ and $M_2$, one obtains a Neumann eigenfunction $u_{\cH}$ of $\cH$, associated with $\nu_6(\cH)=\nu_7(\cH)$, whose nodal set is a closed simple curve around $O$, and whose level lines are displayed in Figure~\ref{F-num-10}; some level lines of $u_{\cH}$ have six connected components, one component near each vertex of the hexagon, so that $\ecp(\cH,\mf{n})$ is false.

\begin{state}\label{S-hex-ecpn}
The $\ecp(\cH,\mf{n})$ is false in $\cE(\nu_1) \oplus \cE(\nu_6)$.
\end{state}%

\begin{figure}[!htb]
  \centering
  \includegraphics[width=7cm]{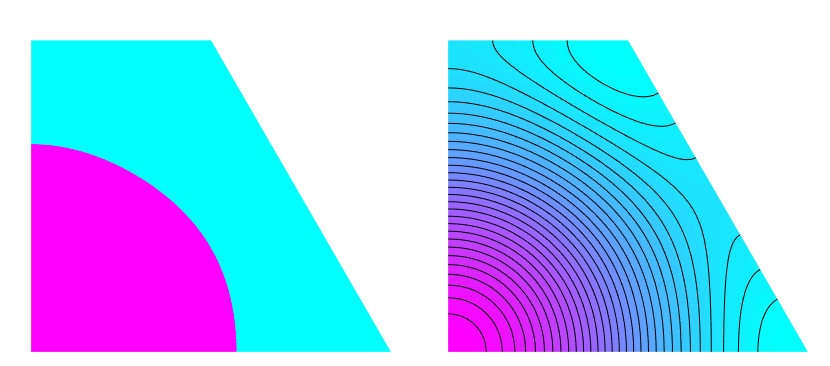}
  \caption{ $(\cR,\mf{nnn})$: nodal set and level lines for $u_3$}\label{F-num-8}
\end{figure}

\begin{figure}[!htb]
  \centering
  \includegraphics[width=5cm]{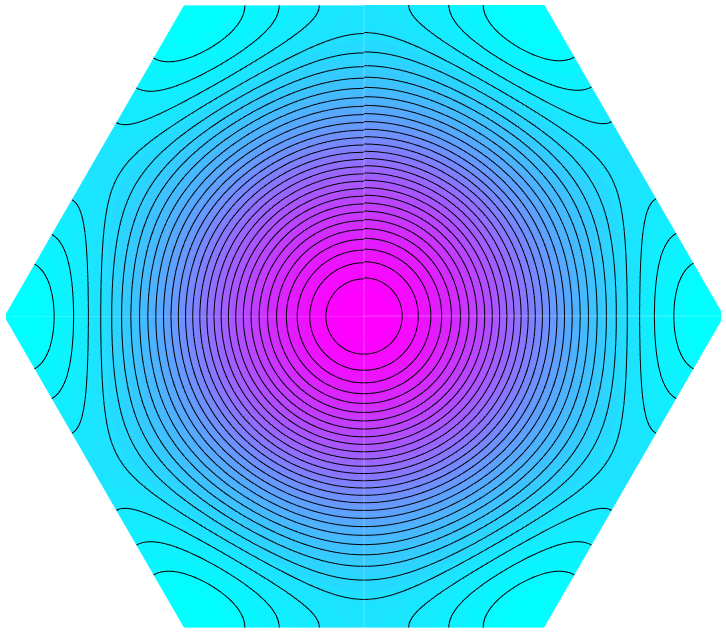}
  \caption{Level lines of $u_{\cH}$}\label{F-num-10}
\end{figure}
\FloatBarrier

\section{Final comments}\label{S-final}

\subsection{Numerical computations}\label{SS-final-1}

In Subsection~\ref{SS-ch-ecp-n}, we used numerical approximations of the first eigenvalues of the problems $(\cR,\mf{nab})$ and $(\cT_h,\mf{abn})$ in order to identify the first eight eigenvalues of $(\cH,\mf{n})$, and to conclude that $\ecp(\cH,\mf{n})$ is false (we also used the fact that some eigenfunctions are known explicitly), see Table~\ref{F-vp-hex-neu}.

\begin{table}[!htb]
  \centering
  \includegraphics[scale=0.9]{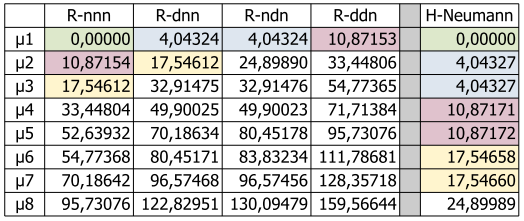}
  \caption{Neumann eigenvalues of $\cH$}\label{F-vp-hex-neu}
\end{table}

We did not find tables providing the first eigenvalues of $(\cH,\mf{n})$ in the literature. We used the symmetries, and computed the eigenvalues of the problems $(\cR,\mf{abn})$ and $(\cT_h,\mf{abn})$ with \textsc{matlab}. We checked the accuracy of our computations in two ways.
\begin{enumerate}
  \item First, using the symmetries, we computed the eigenvalues of $(\cR,\mf{abd})$ and $(\cT_h,\mf{abd})$ in order to obtain the Dirichlet eigenvalues of the hexagon. We then compared the results with the tables in \cite{CurKut1999}, see Table~\ref{F-vp-hex-dir}.
  \item Second, we computed the eigenvalues of $(\cT_h,\mf{abc})$, and compared the results both with explicitly known eigenvalues, and with the tables in \cite{Jon1993}, see Tables~\ref{F-vp-thq-dir} and \ref{F-vp-thq-neu}.
\end{enumerate}

\begin{table}[!htb]
  \centering
  \includegraphics[scale=0.9]{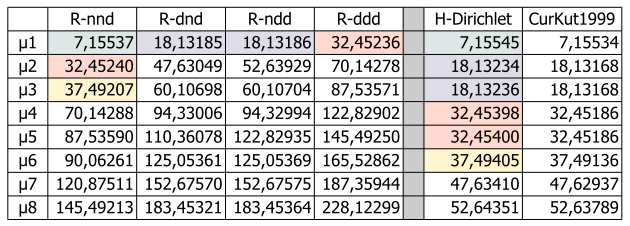}
  \caption{Dirichlet eigenvalues of $\cH$}\label{F-vp-hex-dir}
\end{table}

\begin{table}[!htb]
  \centering
  \includegraphics[scale=1.2]{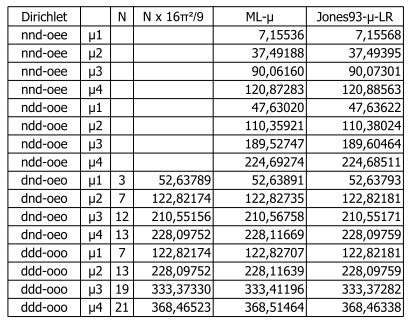}
  \caption{Eigenvalues of $\cT_h$ (Dirichlet on shortest side)}\label{F-vp-thq-dir}
\end{table}

\begin{table}[!htb]
  \centering
  \includegraphics[scale=1.2]{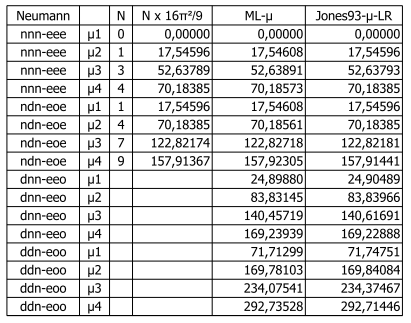}
  \caption{Eigenvalues of $\cT_h$ (Neumann on shortest side)}\label{F-vp-thq-neu}
\end{table}

\begin{remark}\label{R-vp-2}
The tables in \cite{CurKut1999,Jon1993} are organized according to the symmetries, and they provide the \emph{square roots} of the eigenvalues. In \cite{Jon1993}, the labelling of the sides of $\cT_h$ is different from ours: we use $\mf{d,n}$ to indicate the boundary condition on each side, while Jones uses the notation $e,o$ (for even and odd). For the reader's convenience, we indicate both labellings  in the first column of Tables~\ref{F-vp-thq-dir} and \ref{F-vp-thq-neu}. The fourth column of each table contains the eigenvalues which are known explicitly; the fifth column contains our computations. The sixth column of each table contains the values deduced from \cite{Jon1993}, Tables~7--14.
\end{remark}%

\begin{remark}\label{R-vp-4}
Our purpose in this paper is to identify eigenvalues, and their relations with the symmetries, not to find high precision approximations as in \cite{Jon1993,Jon2016}. The approximated values which appear in the tables indicate that the approximations are indeed sufficient to identify the eigenvalues (because we took the symmetries into consideration from the start, and identified multiple eigenvalues).
\end{remark}%

In Subsections~\ref{SS-rhe-ecp-n} and \ref{SS-ch-ecp-d}, we also used numerical approximations of the first and second eigenfunctions of $(\cT_h,\mf{nnd})$ in order to show that the $\ecp(\rhe,\mf{n})$ is false in $\cE(\nu_2)\oplus \cE(\nu_5)$, and that $\ecp(\cH,\mf{d})$ is false in $\cE(\delta_1) \oplus \cE(\delta_6)$.

\subsection{Final remarks}\label{SS-final-2}

\subsubsection{}\label{R-F1} The estimates in Table~\ref{T-DirLU-2} are valid for the regular polygon $\cP_n$ with $n$ sides, inscribed in the circle of radius $1$. The upper bounds get better when $n$ increases, and for $n\ge 9$, they are sufficient to separate $\delta_6(\cP_n)$ from $\delta_5(\cP_n)$. This shows that $\delta_6(\cP_n)$ is a simple eigenvalue for $n\ge 6$, and that an associated eigenfunction $u_6$ arises from the first eigenfunction of a right triangle with smallest angle $\frac{\pi}{n}$, hypotenuse of length $1$, with Dirichlet condition on the smallest side and Neumann condition on the other sides. Equivalently, the eigenfunction $u_6$ arises from a first eigenfunction of an isosceles triangle whose apex angle is $\frac{2\pi}{n}$, with equal sides of length $1$, Dirichlet condition on the smallest side and Neumann condition on the equal sides. Note that $\delta_6(\cD)$ corresponds to the second radial eigenfunction of the disc.\medskip

\subsubsection{}\label{R-F2} Based on our computations, we conjecture that the $\ecp(\cP_n,\mf{a})$ is false for any regular polygon $\cP_n \subset \R^2$ with $n\ge 6$ sides, and $\mf{a} \in \{\mf{d,n}\}$, with some linear combination $u_6 + a u_1$ of a sixth and a first eigenfunctions providing a counterexample with $(n+1)$ nodal domains. Using \cite[Theorem~B]{Miy2013}, one can show that $\ecp(\cP_n,\mf{n})$ is false for $n$ sufficiently large, see \cite{BCH2019}. The simulations show that the first six Dirichlet eigenfunctions of $\cP_n$ look very much like the first six Dirichlet eigenfunctions of the disk $\cD$. \medskip

\subsubsection{}\label{R-F3} The above considerations do not provide any counter-example to the $\ecp$ when the number of sides is $4$ or $5$. It is not clear whether the $\ecp$ is false for the square and for the regular pentagon.  It is not clear either whether the $\ecp$ is false for the disk.

\subsubsection{}\label{R-F4} In the Neumann case, the present paper is also relevant to the investigation of the level lines of Neumann eigenfunctions. Such investigations arise when studying the hot spots conjecture.

\bibliographystyle{plain}

\begin{thebibliography}{}

\end{thebibliography}


\begin{thebibliography}{1}

\bibitem{Al1994} G. Alessandrini.
\newblock Nodal lines of eigenfunctions of the fixed membrane problem in general convex domains.
\newblock Comment. Math. Helvetici 69 (1994) 142--154.

\bibitem{Arn1973} V. Arnold.
\newblock The topology of real algebraic curves (the works of Petrovskii and their development).
\newblock Uspekhi Math. Nauk. 28:5 (1973) 260--262.
\newblock English translation in \cite{Arn2014}.

\bibitem{Arn2011} V. Arnold.
\newblock Topological properties of eigenoscillations in mathematical physics.
\newblock Proceedings of the Steklov Institute of Mathematics 273 (2011) 25--34.

\bibitem{Arn2014} V. Arnold.
\newblock Topology of real algebraic curves (Works of I.G.~Petrovskii and their development). Translated by Oleg Viro. 
\newblock In Collected works, Volume II. Hydrodynamics, Bifurcation theory and Algebraic geometry, 1965--1972.
\newblock Edited by A.B.~Givental, B.A.~Khesin, A.N.~Varchenko, V.A.~Vassilev, O.Ya.~Viro. Springer 2014.

\bibitem{BauRei1978} L. Bauer and E.L. Reiss.
\newblock Cutoff Wavenumbers and Modes of Hexagonal Waveguides.
\newblock SIAM Journal on Applied Mathematics 35:3 (1978) 508--514.


\bibitem{BCH2019} P.~B\'{e}rard, P.~Charron and B.~Helffer.
\newblock Non-boundedness of the number of nodal domains of a sum of eigenfunctions.
\newblock arXiv:1906.03668.

\bibitem{BH2015-tsg} P. B\'{e}rard and B. Helffer.
\newblock Nodal sets of eigenfunctions, Antonie Stern's results revisited.
\newblock S\'{e}minaire de th\'{e}orie spectrale et g\'{e}om\'{e}trie (Grenoble) 32 (2014--2015) 1--37.
\newblock \url{http://tsg.cedram.org/item?id=TSG_2014-2015__32__1_0}\,.

\bibitem{BH2016-lmp} P. B\'{e}rard and B. Helffer.
\newblock Courant-sharp eigenvalues for the equilateral torus, and for the equilateral triangle.
\newblock Letters in Math. Physics 106 (2016) 1729--1789.

\bibitem{BH2018-ecp1} P. B\'{e}rard and B. Helffer.
\newblock On Courant's nodal domain property for linear combinations of eigenfunctions, Part~{I}.
\newblock Documenta Mathematica 23 (2018) 1561--1585.
\newblock arXiv:1705.03731.

 \bibitem{BH2018-sturm} P. B\'{e}rard and B. Helffer.
\newblock Sturm's theorem on zeros of linear combinations of eigenfunctions.
\newblock Expositiones Mathematicae, in press.
\newblock {\textsc{~doi}~} \url{https://doi.org/10.1016/j.exmath.2018.10.002}\,.
\newblock ~arXiv:1706.08247 (expanded version).

\bibitem{BH2018-teqa} P. B\'{e}rard and B. Helffer.
\newblock Level sets of certain Neumann eigenfunctions under deformation of Lipschitz domains. Application to the Extended Courant Property.
\newblock To appear in Annales de la Facult\'{e} des Sciences de Toulouse.
\url{http://afst.cedram.org/}\,.
\newblock arXiv:1805.01335.

\bibitem{CH1953} R. Courant and D. Hilbert.
\newblock Methods of mathematical physics. Vol. 1.
\newblock First English edition. Interscience, New York 1953.

\bibitem{CurKut1999} L.M.~Cureton and J.R.~Kuttler.
\newblock Eigenvalues of the Laplacian on regular polygons and polygons resulting from their disection.
\newblock Journal of Sound and Vibration 220:1 (1999) 83--98.

\bibitem{GZ2003} G. Gladwell and H. Zhu.
\newblock The Courant-Herrmann conjecture.
\newblock ZAMM--Z. Angew. Math. Mech. 83:4 (2003) 275--281.

\bibitem{HOMN1999} T.~Hoffmann-Ostenhof, P.~Michor and N.~Nadirashvili.
\newblock Bounds on the multiplicities of eigenvalues for fixed membranes.
\newblock GAFA, Geom. Func. Anal. 9 (1999) 1169--1188.

\bibitem{Jon1993} R.S.~Jones.
\newblock The one-dimensional three-body problem and selected wave-guide problems: solutions of the two-dimensional Helmholtz equation.
\newblock PhD Thesis, The Ohio State University, 1993. Retyped 2004, available at\\
\url{http://www.hbelabs.com/phd/}\,.

\bibitem{Jon2016} R.S.~Jones.
\newblock Computing ultra-precise eigenvalues of the Laplacian with polygons.
\newblock Adv. Comput. Math. 43 (2017) 1325--1354.
\newblock {arXiv}:1602.08636v1.

\bibitem{Kuz2015} N. Kuznetsov.
\newblock On delusive nodal sets of free oscillations.
\newblock Newsletter of the European Mathematical Society 96 (2015) 34--40.


\bibitem{LaSi2017} R.~Laugesen and B.~Siudeja.
\newblock Triangles and other special domains.
\newblock Chapter~1 in \emph{Shape optimization and spectral theory.} A.~Henrot, ed.
\newblock De Gruyter, Berlin 2017.



\bibitem{Len2016} C.~L\'ena.
\newblock Pleijel's nodal domain theorem for Neumann and Robin eigenfunctions.
\newblock Annales de l'institut Fourier 69:1 (2019) 283--301.
\newblock arXiv:1609.02331.

\bibitem{LeWe1986} H.~Levine and H.~Weinberger.
\newblock Inequalities between Dirichlet and Neumann eigenvalues.
\newblock Arch. Rational Mech. Anal. 94 (1986) 193--208.

\bibitem{LoRo2017} V.~Lotoreichik and J.~Rohledder.
\newblock Eigenvalue inequalities for the Laplacian with mixed boundary conditions.
\newblock J. Differential Equations 263 (2017) 491--508.

\bibitem{Miy2013} Y.~Miyamoto.
\newblock A planar convex domain with many isolated ``hot spots''
on the boundary.
\newblock Japan J. Indust. Appl. Math. 30 (2013) 145--164.

\bibitem  {Ple1956} {\AA}.~Pleijel.
\newblock Remarks on Courant's nodal theorem.
\newblock Comm. Pure. Appl. Math. 9 (1956) 543--550.

\bibitem{Siu2016} B.~Siudeja.
\newblock On mixed Dirichlet-Neumann eigenvalues of triangles.
\newblock Proc. Amer. Math. Soc. 144 (2016) 2479--2493.

\bibitem{Vir1979} O. Viro.
\newblock Construction of multi-component real algebraic surfaces.
\newblock Soviet Math. dokl. 20:5 (1979) 991--995.

\end{thebibliography}

\end{document}